\documentclass[11pt]{amsart}
\usepackage{amssymb}
\usepackage{latexsym} % to use some math symbols
\usepackage{srcltx} %jump between tex and dvi, used for kile
\usepackage{esint}
\usepackage{color}
\def\nc{\newcommand}

\def\ep{\epsilon}

\def\om{\omega}

 \def\Om{\Omega}
\def\ra{\rightarrow}

\nc\pa{\partial}

\nc\CC{\mathbb{C}}
\nc\RR{\mathbb{R}}
\nc\QQ{\mathbb{Q}}
\nc\ZZ{\mathbb{Z}}
\nc\NN{\mathbb{N}}

\nc\m[1]{\left| #1\right|}
\nc\norm[1]{\left\| #1\right\|}
\newtheorem{theorem}{Theorem}[section]
\newtheorem{lemma}[theorem]{Lemma}
\newtheorem{corollary}[theorem]{Corollary}
\newtheorem{proposition}[theorem]{Proposition}
\newtheorem{definition}[theorem]{Definition}% Use {\rm ...}
\newtheorem{remark}[theorem]{Remark}        % Use {\rm ...}
\numberwithin{equation}{section}

\begin{document}

\title[Quasilinear equations with source terms on  Carnot groups]
{Quasilinear equations with source terms on  Carnot groups}

\author[Nguyen Cong Phuc]{Nguyen Cong Phuc$^{*}$}
\address{Department of Mathematics, Louisiana State University, 303 Lockett Hall,
Baton Rouge, LA  70803, USA}
\email{pcnguyen@math.lsu.edu}

\author[Igor E. Verbitsky]{Igor E. Verbitsky$^{**}$}
\address{Department of Mathematics,
University of Missouri,
Columbia, MO 65211, USA}
\email{verbitskyi@missouri.edu}

\thanks{2010 Mathematics Subject Classification: Primary 35H20, Secondary 35A01, 20F18}
\thanks{$^*$Supported in part by NSF Grant DMS-0901083}
\thanks{$^{**}$Supported in part by NSF Grant DMS-0901550}

\begin{abstract} 
In this paper we give necessary and sufficient conditions for the existence of solutions 
to quasilinear equations of Lane--Emden type with measure data on a Carnot group $\mathbb G$ of  arbitrary step. The quasilinear part 
involves operators of the $p$-Laplacian type $\Delta_{\mathbb G,\,p}\,$, $1<p<\infty$. These results are based on 
new a priori estimates of solutions in terms of  nonlinear potentials of Th. Wolff's type. As a consequence, we characterize completely removable singularities, 
and prove a Liouville type theorem 
for supersolutions of quasilinear equations with source terms which has been 
known only for equations involving the sub-Laplacian ($p=2$) on the Heisenberg group. 
\end{abstract}

\maketitle

\section{Introduction}\label{Introduction}

In this paper we study the solvability problem and pointwise estimates of solutions  for a class of quasilinear  Lane-Emden type equations with measure data 
on Carnot groups of  arbitrary step. A complete characterization of removable singularities  for the 
corresponding homogeneous equations as well as a Liouville type theorem for supersolutions 
 will also be obtained  as a consequence. 

The basic setting of our study is a given  Carnot group $\mathbb{G}$  of step $r\geq 1$, i.e., a connected and simply connected stratified nilpotent Lie group
whose Lie algebra $\mathcal{G}$ admits a stratification $\mathcal G=V_1\oplus V_2 \oplus\cdots\oplus V_r$ and is generated via commutations 
by its first (horizontal) layer $V_1$ (see Sect. \ref{Cgroup}). Given a basic $\{X_j\}_{j=1}^{m}$ of  $V_1$, the associated $p$-Laplacian operator $\Delta_{\mathbb G,\,p}\,$, $1<p<\infty$, is defined by
$$\Delta_{\mathbb G,\,p}\, u=\sum_{i=1}^{m}X_{i}(|Xu|^{p-2}X_{i}u),$$  
where $$Xu=X_{1}uX_{1}+X_{2}uX_{2}+\cdots+X_{m}uX_{m},$$ 
is the {\it horizontal gradient} of $u$ with length $|Xu|=\left(\sum_{i=1}^{m}|X_{i}u|^2\right)^{1/2}.$

We study the following Lane-Emden type equation on a bounded open set $\Om\subset \mathbb{G}$:  
\begin{eqnarray}\label{L-E}
\left\{\begin{array}{rcl}
-\Delta_{\mathbb G,\,p}\, u&=&u^{q}+\om \quad{\rm in~}\Om,\\
u&=&0 \quad {\rm on~} \partial\Om,
\end{array}
\right.
\end{eqnarray}
where $q>p-1>0$, and $\om$ is a given nonnegative finite measure on $\Om$. 
 Our objective is to obtain necessary and sufficient conditions on 
the measure  $\om$ for the existence of solutions to \eqref{L-E},  and to give a complete characterization of removable singularities for the corresponding homogeneous equation: 
\begin{equation}\label{L-E-homo}
-\Delta_{\mathbb G,\,p}\, u=u^{q} \quad{\rm in~}\Om.
\end{equation}

Equations similar to \eqref{L-E} in the entire group $\mathbb G$ are also considered with applications to Liouville type theorems for the differential inequality
\begin{equation}\label{diffin}
-\Delta_{\mathbb G,\,p}\, u\geq u^{q}\quad{\rm in~}\mathbb G.
\end{equation}

Such problems have been studied in depth in our previous work  \cite{PV1}, \cite{PV2} in the standard Euclidean setting; see also earlier work in \cite{BP}, \cite{AP}, \cite{BV1}, and \cite{BV2}. However, in the setting of Carnot groups, the failure of the Besicovitch covering lemma (see \cite{SW}, \cite{KR}) and the lack of  a perfect dyadic grid of cubes cause major difficulties. We observe that even in the setting of the Heisenberg group, the simplest model
of a non-commutative Carnot group, Liouville type theorems for the differential inequality \eqref{diffin} are known only in the sub-Laplacian case, i.e.,  $p=2$ (see \cite{GL}, \cite{BCC}, \cite{PVe}).

A substantial part of our study of   \eqref{L-E} is devoted to integral inequalities for both linear and nonlinear potential operators 
and their discrete analogues  over  ``approximate" dyadic grids of cubes constructed in \cite{SW} and \cite{Chr} in the general setting of 
homogeneous spaces.

For each $\alpha>0$, the Bessel potential of a locally integrable function  $f$ in this setting is defined by
$${\mathbf G}_\alpha (f)(x)={\mathbf G}_\alpha * f(x)=\int_{\mathbb G} {\mathbf{G}}_\alpha(y^{-1} x) f(y) dy, \qquad x\in \mathbb G,  $$
where $\mathbf{G}_\alpha$ is the Bessel kernel of order $\alpha$ on $\mathbb G$ given by
\begin{equation}\label{BesselK}
\mathbf{G}_\alpha(x)=\frac{1}{\Gamma(\alpha/2)} \int_{0}^{\infty} t^{\alpha/2-1} e^{-t} h(x, t) dt. 
\end{equation}
In \eqref{BesselK} $h(x, t)$ is the heat kernel associated with the sub-Laplacian $\Delta_{\mathbb G}=\Delta_{\mathbb G, \, 2}$ whose basic properties
can be found in  \cite{Fol}, \cite{VSC}.
We also write 
$${\mathbf G}_\alpha (fd\mu)(x)={\mathbf G}_\alpha *(fd\mu)(x)=\int_{\mathbb G} {\mathbf{G}}_\alpha(y^{-1} x) f(y) d\mu(y), \qquad x\in \mathbb G,  $$
for each locally $\mu$-integrable function $f$. 

When dealing with  solutions on the entire group $\mathbb G$ and Liouville type theorems we need to use another linear potential, the Riesz potential. For each $0<\alpha<M$ and 
$f\in L^1_{\rm loc}(\mathbb G)$, it is defined by  
$${\mathbf I}_\alpha (f)(x)={\mathbf I}_\alpha * f(x)=\int_{\mathbb G}  \frac{f(y)}{d_{cc}(x, y)^{M-\alpha}} dy, \qquad x\in \mathbb G,  $$
where $d_{cc}$ is the Carnot-Carath\'eodory distance on $\mathbb G$, and $M$ is the homogeneous dimension of $\mathbb G$ (see Sect. \ref{Cgroup}). 

Associated with the kernel ${\mathbf G}_\alpha$ is the Bessel capacity ${C}_{\alpha, \, s }(\cdot)$, $s>1$, defined
by (see \cite{AH}, Sec. 2.6, in the Euclidean case) 
$${C}_{\alpha, \, s } (E)=\inf\{\norm{f}_{L^s(\mathbb G)}^{s}: {\mathbf G}_\alpha (f)\geq 1 {\rm ~on~} E, ~ f \in L^s(\mathbb G), ~ f\geq 0\}$$
for each compact set $E\subset \mathbb G$. Similarly,  the Riesz capacity $\dot{{C}}_{\alpha, \, s } (\cdot)$, $0<\alpha<M$, $s>1$, is defined,  
for a compact set $E\subset \mathbb G$, by
$$\dot{{C}}_{\alpha, \, s } (E)=\inf\{\norm{f}_{L^s(\mathbb G)}^{s}: {\mathbf I}_\alpha (f)\geq 1 {\rm ~on~} E, ~ f \in L^s(\mathbb G), ~ f\geq 0\}.$$

These capacities will play an essential role  in our characterizations of the existence of solutions and removable singularities, as well as Liouville type theorems for the Lane-Emden type equation. We will also need the following dual definition of these capacities (see \cite[Theorem 2.10]{Lu}; 
  \cite[Theorem 2.2.7]{AH} in the Euclidean case):
\begin{equation}\label{dualB}
{C}_{\alpha, \, s } (E)=\sup_{\mu\in \mathcal{M}^+(E), \,  \mu \not= 0} \left( \frac{\mu(E)}{\norm{{\rm\bf G}_{\alpha}*\mu}_{L^{\frac{s}{s-1}}(\mathbb G)}}\right)^s,
\end{equation}
and similarly,
\begin{equation}\label{dualR}
\dot{{C}}_{\alpha, \, s } (E)=\sup_{\mu\in \mathcal{M}^+(E),  \, \mu \not= 0} \left( \frac{\mu(E)}{\norm{{\rm\bf I}_{\alpha}*\mu}_{L^{\frac{s}{s-1}}(\mathbb G)}}\right)^s,
\end{equation}
where $\mathcal{M}^+(E)$ denotes the set of all nonnegative measures supported on $E$.
  
The nonlinear potential  we use  below is the (truncated) Wolff's potential $\mathbf{W}_{\alpha,\,p}^{R}$ originally introduced in \cite{HW}. 
In our setting, for  $\alpha>0$,  $p>1$, and $0<R\leq \infty$,   it is defined for each nonnegative measure $\mu$ on $\mathbb G$   by
\begin{equation*}
\mathbf{W}_{\alpha,\,p}^{R}\mu(x)=\int_{0}^{R}\Big[\frac{\mu(B_{t}(x))}{t^{M-\alpha p}}\Big]^{\frac{1}{p-1}}\frac{dt}{t},
\quad\quad x\in \mathbb G,
\end{equation*}
where $B_t(x)$ is the Carnot-Carath\'eodory ball centered at $x$ of radius $t$ (see Sect. \ref{Cgroup}).

For our purpose  we introduce the following notion of solutions  for $p$-Laplace equations with general measure as data. This will serve as an efficient 
substitution for the notion of renormalized solutions introduced in \cite{DMOP} in the Euclidean setting.
 
\begin{definition} \label{DEFEQ}
For a nonnegative finite measure $\mu$ on $\Om$, we say that $u$ is a solution to
\begin{eqnarray}\label{eqme}
\left\{\begin{array}{rcl}
-\Delta_{\mathbb G,\, p}\, u&=&\mu \quad {\rm in~} \Om,\\
u&=&0 \quad {\rm on~} \partial\Om,
\end{array}
\right.
\end{eqnarray}
in the potential theoretic sense if $u$ is $p$-superharmonic in $\Om$,  $\min\{u,k\}\in S_{0}^{1,\, p}(\Om)$ for every $k>0$, 
$u$ satisfies a pointwise bound
\begin{equation}\label{Kconst}
u(x)\leq A\,\mathbf W_{1,\,p}^{2{\rm diam}(\Om)}(x), \qquad\forall  x\in \Om, 
\end{equation}
and for every $\varphi\in C^{\infty}_{0}(\Om)$ one has
$$\int_{\Om}|Xu|^{p-2}Xu \cdot X\varphi dx=\int_{\Om}\varphi d\mu.$$
\end{definition}

From this definition we see right away that potential theoretic solutions to \eqref{eqme} are also distributional solutions. However, 
the converse is not necessarily true as easily seen by a simple example (see \cite{Kil}).
The existence of potential theoretic solutions to \eqref{eqme} will be obtained in Corollary \ref{corexist},  whereas their uniqueness is unknown
 even in the Euclidean setting.

In Definition \ref{DEFEQ} the notation $S_{0}^{1,\, p}(\Om)$ stands for the completion of $C^{\infty}_{0}(\Om)$ under the 
norm of the horizontal Sobolev space $S^{1,\, p}(\Om)$ (see Sect. \ref{sec3}), and 
in \eqref{Kconst}, $A$ is a universal constant independent of $x, u, \mu,$ and $\Om$. 
For the notion of $p$-superharmonic functions on Carnot groups see Sect. \ref{sec3}.
We are now ready to state the first result of the paper. 

\begin{theorem}\label{Dmain1}
Let $p>1$, $q>p-1$, and $R={\rm diam}(\Omega)$.
Suppose that $\omega$ is a nonnegative finite measure on $\Omega$ such that
${\rm supp}(\om)\Subset\Om$. If the equation
\begin{equation}\label{DMA1}
\left\{\begin{array}{rcl}
-\Delta_{\mathbb{G},\,p}\, u&=&u^{q}+ \omega \quad{\rm in~} \Omega,\\
u&=&0 \quad {\rm on~} \partial\Omega
\end{array}
\right.
\end{equation}
has a nonnegative $p$-superharmonic distributional solution  $u\in L^{q}(\Omega)$,  then there exists a
constant $C>0$ such that  statements {\rm(i)}--{\rm (v)} below hold true.
\begin{itemize}
\item[(i)] The inequality
\begin{equation}\label{DMA2'}
\int_{\mathbb{G}}\mathbf{G}_{p}(f)^{\frac{q}{q-p+1}} d\omega\leq C \, \int_{\mathbb{G}}f^{ \frac{q}{q-p+1}}dx
\end{equation}
holds for all $f\in L^{\frac{q}{q-p+1}}$, $f\geq 0$.
\item[(ii)] For every compact set $E\subset\Om$,
$$\omega(E)\leq C\,{C}_{p,\,\frac{q}{q-p+1}}(E).$$
\item[(iii)] The inequality
\begin{equation}\label{DMA2}
\int_{\mathbb{G}}[{\rm\bf W}^{2R}_{1,\,p}(g d\omega) (x)]^{q} \,
dx\leq C \, \int_{\mathbb{G}}g^{ \frac{q}{p-1}}d\omega
\end{equation}
holds for all $g\in L^{\frac{q}{p-1}} (d \omega)$, $g\geq 0$.
\item[(iv)] The inequality
\begin{equation}\label{DMA4}
\int_{B}[{\rm\bf W}_{1,\,p}^{2R}\omega_{B}(x)]^{q} \, dx\leq C \, \omega(B)
\end{equation}
holds for all Carnot-Carath\'eodory balls $B\subset\mathbb{G}$.
\item[(v)] For all $x\in\Omega$,
\begin{equation}\label{DMA5}
{\rm\bf W}_{1,\, p}^{2R}[({\rm\bf W}_{1,\, p}^{2R}\omega)^{q}](x)\leq C  \,
{\rm\bf W}_{1,\, p}^{2R}\omega(x).
\end{equation}
\end{itemize}

Conversely, there exists a constant $C_{0}=C_{0}(M,p,q)>0$ such
that if any one of the  statements {\rm (i)}--{\rm (v)} holds with
$C\leq C_{0}$ then equation (\ref{DMA1}) has a  nonnegative potential theoretic solution
$u\in L^{q}(\Omega)$ for any nonnegative finite measure $\om$. Moreover, $u$ satisfies
the following pointwise estimate
$$u\leq \kappa\,{\rm\bf W}^{2R}_{1,\, p}\omega.$$
\end{theorem}

Our second result is about  removable singularities of solutions to homogeneous equations, which is in fact a
consequence of Theorem \ref{Dmain1}.   

\begin{theorem}\label{removeforp} Let $q>p-1>0$ and  let $E$ be a compact subset of $\Om$.
Then any solution $u$ to the problem
\begin{equation}\label{omE}
\left\{\begin{array}{c}
 u {\rm ~is~} p{\text-}{\rm superharmonic~in~}\Om\setminus E,\\
u\in L^{q}_{\rm loc}(\Om\setminus E), \quad u \ge 0,\\
-\Delta_{\mathbb{G},\,p}\, u=u^{q} \quad {\rm in} \quad\mathcal{D}'(\Om\setminus E)
\end{array}
\right.
\end{equation}
is also a solution to a similar problem with $\Om$ in place of $\Om\setminus E$
if and only if $${C}_{p,\,\frac{q}{q-p+1}}(E)=0.$$
\end{theorem}

The proof of  Theorems \ref{Dmain1} and \ref{removeforp} will be given at the end of Sect. \ref{sec4}.  In  case the bounded domain 
$\Om$ in Theorem \ref{Dmain1} is replaced by the whole group $\mathbb G$, then Riesz potentials and the corresponding Riesz capacity must be used, 
and  we have the following  result.

\begin{theorem}\label{Gmain1}
Let $1<p<M$, $q>p-1$ and let  $\omega$ be a nonnegative locally finite measure on $\mathbb G$. If the equation
\begin{equation}\label{global}
\left\{\begin{array}{rcl}
-\Delta_{\mathbb{G},\,p}\, u&=&u^{q}+ \omega \quad{\rm in~} \mathbb G,\\
\inf_{\mathbb G} u&=&0
\end{array}
\right.
\end{equation}
has a nonnegative $p$-superharmonic distributional solution  $u\in L^{q}_{\rm loc}(\mathbb G)$,  then there exists a
constant $C>0$ such that  statements {\rm(i)}--{\rm (vi)} below hold true.
\begin{itemize}
\item[(i)] For every compact set $E\subset\mathbb G$,
$$\int_{E} u^q\, dx \leq C\,\dot{{C}}_{p,\,\frac{q}{q-p+1}}(E).$$
\item[(ii)] The inequality
\begin{equation*}
\int_{\mathbb{G}}\mathbf{I}_{p}(f)^{\frac{q}{q-p+1}} d\omega\leq C \, \int_{\mathbb{G}}f^{ \frac{q}{q-p+1}}dx
\end{equation*}
holds for all $f\in L^{\frac{q}{q-p+1}}$, $f\geq 0$.
\item[(iii)] For every compact set $E\subset\mathbb G$,
$$\omega(E)\leq C\,\dot{{C}}_{p,\,\frac{q}{q-p+1}}(E).$$
\item[(iv)] The inequality
\begin{equation*}
\int_{\mathbb{G}}[{\rm\bf W}^{\infty}_{1,\,p}(g d\omega) (x)]^{q} \,
dx\leq C \, \int_{\mathbb{G}}g^{ \frac{q}{p-1}}d\omega
\end{equation*}
holds for all $g\in L^{\frac{q}{p-1}}(d \omega)$, $g\geq 0$.
\item[(v)] The inequality
\begin{equation*}
\int_{B}[{\rm\bf W}_{1,\,p}^{\infty}\omega_{B}(x)]^{q} \, dx\leq C \, \omega(B)
\end{equation*}
holds for all Carnot-Carath\'eodory balls $B\subset\mathbb{G}$.
\item[(vi)] For all $x\in\Omega$,
\begin{equation*}
{\rm\bf W}_{1,\, p}^{\infty}[({\rm\bf W}_{1,\, p}^{\infty}\omega)^{q}](x)\leq C  \,
{\rm\bf W}_{1,\, p}^{\infty}\omega(x).
\end{equation*}
\end{itemize}

Conversely, there exists a constant $C_{0}=C_{0}(M,p,q)>0$ such
that if any one of the  statements {\rm (ii)}--{\rm (vi)} holds with
$C\leq C_{0}$ then equation (\ref{global}) has a  nonnegative $p$-superharmonic solution
$u\in L^{q}_{\rm loc}(\mathbb G)$. Moreover, $u$ satisfies
the following pointwise two-sided estimate
$$\kappa_1\,{\rm\bf W}^{\infty}_{1,\, p}\omega \leq u\leq \kappa_2\,{\rm\bf W}^{\infty}_{1,\, p}\omega.$$
\end{theorem}

Theorem \ref{Gmain1} yields the following Liouville type theorem. 
We observe that for $p\not=2$ this Liouville type theorem is new even in the Heisenberg group. 
For $p=2$, as mentioned earlier,  such a result was obtained in \cite{GL}, \cite{BCC}, and \cite{PVe} in the setting of the Heisenberg group.
However, the approach of using test functions and integration by parts in these papers does not seem to work in the general setting 
of Carnot groups of arbitrary step.

\begin{corollary}\label{LiouvilleT} If $q\leq \frac{M(p-1)}{M-p}$, then the inequality
$-\Delta_{\mathbb G,\, p}\, u \geq u^q$ admits no nontrivial nonnegative $p$-superharmonic distributional solutions
in $\mathbb G$.
\end{corollary}

The proofs of Theorem \ref{Gmain1} and Corollary \ref{LiouvilleT} will be given in Sect. \ref{secv}. 

\section{Preliminaries on Carnot groups}\label{Cgroup}
Let $\mathbb G$ be a Lie group, i.e., a differentiable manifold endowed with a group structure such that the map
$\mathbb G\times\mathbb G\rightarrow\mathbb G$ defined by $(x,y)\mapsto xy^{-1}$ is $C^\infty$. Here $y^{-1}$ is
the inverse of $y$ and $xy^{-1}$ denotes the group multiplication of $x$ by $y^{-1}$. We will denote by
$$L_{x_{0}}(x)=x_{0}x,\quad\quad R_{x_{0}}(x)=xx_{0},$$
respectively, the left- and right-translations on $\mathbb G$. A vector field $X$ on $\mathbb G$ is called
left-invariant if for each $x_{0}\in\mathbb G$,
$$dL_{x_0}(X(x))=X(x_0 x)$$
for all $x\in\mathbb G$, i.e., $dL_{x_0}\circ X=X\circ L_{x_0}$. Here $dL_{x_0}$ is the differential
of $L_{x_0}$. Under the Lie bracket operation on vector fields, the set of left-invariant
vector fields on $\mathbb G$ forms a Lie algebra called the Lie algebra of $\mathbb G$ and is
denoted by  $\mathcal{G}$.
Note that we can identify $\mathcal{G}$ with the tangent space $\mathbb G_{e}$
to $\mathbb G$ at the {\it identity} $e\in \mathbb{G}$  via the isomorphism $\alpha: \mathcal{G}\rightarrow \mathbb G_{e}$
defined by $\alpha(X)=X(e)$ and thus ${\rm dim}\, \mathcal G={\rm dim}\, \mathbb G= N$, the {\it topological dimension}
of $\mathbb G$.

A Carnot group $\mathbb G$ of step $r$ is a connected and simply connected Lie group whose Lie algebra $\mathcal G$ admits a
nilpotent stratification of step $r$, i.e., $\mathcal G=V_1\oplus V_2 \oplus\cdots\oplus V_r$ with
$[V_{1}, V_{i}]=V_{i+1}$ for $i=1, \dots, r-1$, $V_r\not=\{0\}$ and $[V_1, V_r]={0}$, where $[\cdot,\cdot]$ denotes the Lie bracket.

Let  $\{X_{j}\}_{j=1}^{m}$ be a basis for the first layer $V_{1}$ (also called the {\it horizontal layer}) of
$\mathcal{G}$. Then for $2\leq i\leq r$, we can choose
a basis $\{X_{ij}\}$, $1\leq j\leq {\rm dim}(V_{i})$, for $V_{i}$ consisting of
commutators of length $i$. In particular, $X_{1j}=X_{j}$ for $j=1,\dots, m$, and
$m={\rm dim}(V_1)$. We then define an inner product $<\cdot \ , \cdot>$ on $\mathcal G$ by declaring the $X_{ij}$'s to be orthonormal.
Since $\mathbb G$  is connected and  simply connected, the exponential map $\rm exp$ is a global diffeomorphism from $\mathcal{G}$ onto
$\mathbb G$ (see \cite{VSC}, \cite{Va}). Thus
for each $x\in \mathbb G$, there is a unique $\hat{x}=(x_{ij})\in \RR^N$, $1\leq i\leq
r$, $1\leq j\leq {\rm dim}(V_{i})$, and $N=\sum_{i=1}^{r}{\rm dim}(V_{i})$, the
 topological dimension of $\mathbb{G}$, such that
$$x={\rm exp}\Big(\sum x_{ij}X_{ij}\Big).$$
Thus the maps $\phi_{ij}: \mathbb G\rightarrow \RR$, $1\leq i\leq r$, $1\leq j\leq {\rm dim}(V_i)$, defined by
$$\phi_{ij}(x)=x_{ij}\quad {\rm for~} x={\rm exp}\Big(\sum x_{ij}X_{ij}\Big),$$
form a system of global coordinates on $\mathbb G$ which are called the {\it exponential coordinates}. Henceforth we will always use
these  coordinates and simply write
$$x=(\phi_{ij}(x))=(x_{ij})\quad {\rm for~} x={\rm exp}\Big(\sum x_{ij}X_{ij}\Big).$$
Such an identification of $\mathbb G$ with its Lie algebra is justified by the Baker-Cambell-Hausdorff formula
(see, e.g., \cite{Va})
$${\rm exp}\Big(\sum x_{ij}X_{ij}\Big){\rm exp}\Big(\sum y_{ij}X_{ij}\Big)={\rm exp}\Big[H\Big(\sum x_{ij}X_{ij},
\sum y_{ij}X_{ij}\Big)\Big],$$
where $H(X,Y)=X+Y+\frac{1}{2}[X,Y]+\cdots$ with the dots indicating a finite linear combination of terms
containing commutators of order two and higher. If we define a group law $*$ on $\mathcal G$ by
$$X*Y=H(X,Y)$$
then the group $\mathbb G$ can be identified with $(\mathcal G,*)$ via the exponential coordinates.
Note that from the Baker-Cambell-Hausdorff formula we have
$$\phi_{ij}(x_0 x)=\phi_{ij}(x_0)+\phi_{ij}(x)+ P_{ij}(x_0,x),$$
where $P_{ij}(x_0,x)$ depends only on the coordinates $\phi_{kl}(x_0)$ and $\phi_{kl}(x)$ with $k<i$.
Thus the determinant of $dL_{x_0}$ is equal to $1$, and the same properties hold for the right translation
$R_{x_0}$ and its differential $dR_{x_0}$ as well. It follows that Lebesgue measure on $\mathcal G$
is lifted via the exponential mapping $\rm {exp}$ to a bi-invariant Haar measure on $\mathbb G$, which we will denote
by  $dx$.

For a given function $f: \mathbb G\rightarrow \RR$, the action of $X\in \mathcal G$ on $f$ is specified by the equation
$$X f(x)= \lim_{t\rightarrow 0}\frac{f(x \,{\rm exp}(tX))-f(x)}{t}=\frac{d}{dt}f(x\, {\rm exp}(tX))|_{t=0}.$$

For $t>0$, we define the dilation $\delta_{t}: \mathbb G\rightarrow \mathbb G$ by
$$\delta_{t}(x)=(t^{i}\phi_{ij}(x))$$
whose Jacobian determinant is everywhere equal to $t^{M}$, where
$$M=\sum_{i=1}^{r}i{\rm ~dim}(V_i)$$
is the {\it homogeneous dimension} of $G$. A {\it homogeneous norm} $|\cdot|$ on $\mathbb G$ is defined by
$$|x| =\Big( \sum \m{\phi_{ij}(x)}^{2r!/i}\Big)^{1/2r!},$$
which obviously satisfies $|\delta_{t}(x)|=t|x|$ and $|x^{-1}|=|x|$. This homogeneous norm
generates a quasi-metric $\rho(x,y)=|x^{-1}y|$ equivalent to the {\it Carnot-Carath\'eodory
metric} $d_{cc}$ on $\mathbb G$ (see \cite{NSW}, \cite{VSC}). Here
$$d_{cc}(x,y)=\inf_{\gamma}\int_{a}^{b}\sqrt{<\dot{\gamma}(t),\dot{\gamma}(t)>}dt,$$
where the infimum is taken over all curves $\gamma:[a,b]\rightarrow \mathbb G$ such that
$\gamma(a)=x$, $\gamma(b)=y$ and $\dot{\gamma}(t)\in V_1$ for all $t$. Such a curve is called a {\it horizontal curve}
connecting $x, y\in \mathbb G$. By Chow-Rashevsky's accessibility theorem (see \cite{Cho}, \cite{Ra}), any two points $x, y\in\mathbb G$
can be joined by a horizontal curve of finite length and hence $d_{cc}$ is  a left-invariant metric on $\mathbb G$.
We will denote by
$$B_{R}(x)=\{y\in \mathbb G: d_{cc}(x,y)<R\}$$
the Carnot-Carath\'eodory metric ball centered at $x$ with radius $R$. Note that there is $c=c(\mathbb G)$ such that
$$|B_{R}(x)|=c R^{M},$$
where for a Borel set $E\subset\mathbb G$ we write $|E|$ for $\int_{E}dx$. Moreover, by homogeneity and left-invariance we have
$$|\delta_{t}(E)|=t^{M}|E|,\quad d(\delta_{t}(x))=t^{M}dx,$$
and for $x, x', y\in \mathbb G$,
$$\quad d_{cc}(yx,yx')=d_{cc}(x,x'),\quad B_{R}(x)=x B_{R}(e).$$

\section{$p$-superharmonic functions on Carnot groups}\label{sec3}

Let $p> 1$ and let $\Om$ be an open set in $\mathbb G$. Recall from the previous section that
$X=(X_1,X_2,\dots,X_m)=(X_{11},X_{12},\dots,X_{1m})$ is an 
orthonormal basic for the first layer $V_1$ of $\mathcal G$. The horizontal Sobolev space $S^{1,\,p}(\Om)$ associated with the system $X$ is defined
by
$$S^{1,\,p}(\Om)=\{u\in L^{p}(\Om): X_{i}u\in L^{p}(\Om),~i=1,\dots,m\},$$
where $X_{i}u$ is understood in the sense of distributions, i.e.,
$$X_{i}u(\varphi)=-\int_{\Om}u X_{i}\varphi dx$$
for every $\varphi\in C_{0}^{\infty}(\Om)$. It is a Banach space equipped with the norm
$$||u||_{S^{1,\,p}(\Om)}=\Big(\int_{\Om}(|u|^p+|Xu|^p)\Big)^{\frac{1}{p}}.$$
The corresponding local Sobolev space  $S_{\rm loc}^{1,\,p}(\Om)$
is defined similarly, with $L^p_{\rm loc}(\Om)$ in place of $L^{p}(\Om)$. We will denote by $S^{1,\,p}_{0}(\Om)$
the completion of $C_{0}^{\infty}(\Om)$ under the norm  $||\cdot||_{S^{1,\,p}(\Om)}$.

Recall that for a smooth function $u$ on $\mathbb G$, the {\it $p$-Laplacian} of $u$ is defined by
$$\Delta_{\mathbb G,\,p}\, u=\sum_{i=1}^{m}X_{i}(|Xu|^{p-2}X_{i}u),$$
where $Xu=X_{1}uX_{1}+X_{2}uX_{2}+\cdots+X_{m}uX_{m}$ is  the {\it horizontal gradient} of $u$, and $|Xu|^{2}=\sum_{i=1}^{m}|X_{i}u|^2$.
A  function $u \in S_{\rm loc}^{1,\,p}(\Om)$ is said to be a weak solution to
\begin{equation}\label{weaksol}
\Delta_{\mathbb G,\,p}\, u=0
\end{equation}
if
$$\int_{\Om}|Xu|^{p-2}Xu\cdot X\varphi dx=0$$
for every $\varphi\in C_{0}^{\infty}(\Om)$. Here $Xu\cdot X\varphi=\sum_{i=1}^{m}X_iu X_i \varphi $.
It is known that every weak solution to \eqref{weaksol} has a continuous
representative (see \cite{TW}, \cite{HKM}), and such continuous solutions are called {\it $p$-harmonic} functions on $\Om$.
On the other hand, if $u\in S^{1,\,p}_{\rm loc}(\Om)$ and
$$\int_{\Om}|Xu|^{p-2}Xu\cdot X\varphi dx\geq0$$
for every $\varphi\in C_{0}^{\infty}(\Om)$, $\varphi\geq 0$ then $u$ is called a {\it supersolution} to \eqref{weaksol}.

A lower semicontinuous function $u:\Om\rightarrow(-\infty,\infty]$
is called {\it $p$-superharmo-nic} if $u$ is not identically infinite in each component of $\Om$, and if for all open sets $D$ such that
$\overline{D}\subset\Om$, and all functions $v\in C(\overline{D})$, $p$-harmonic in $D$, it follows that $v\leq u$ on $\partial D$
implies $v\leq u$ in $D$.

The following fundamental connection between supersolutions to \eqref{weaksol} and $p$-superharmonic functions can be found
in \cite{TW}.

\begin{proposition}\label{pro2.1} Let $u\in S^{1,\,p}_{\rm loc}(\Om)$ be a supersolution to \eqref{weaksol}. Let
$$\underline{u}(x)={\rm ess}\liminf_{y\rightarrow x} u(y).$$
Then $\underline{u}$ is $p$-superharmonic and $u=\underline{u}$ a.e.
\end{proposition}
From this proposition it follows that we may assume all supersolutions to be lower semicontinuous. Therefore a function $u$ is a supersolution to
\eqref{weaksol} if and only if $u$ is $p$-superharmonic and belongs to $S^{1,\,p}_{\rm loc}(\Om)$.

Note that a $p$-superharmonic function $u$ does not
necessarily belong to $S^{1,\,p}_{{\rm loc}}(\Om)$, but its
truncation $\min\{u,k\}$ does for every integer $k$. Using this we set
\begin{eqnarray*}
Xu=\lim_{k\ra\infty} \, X [ \, \min\{u,k\}],
\end{eqnarray*}
defined a.e. If either $u\in L^{\infty}(\Om)$ or $u\in S^{1,\,1}_{{\rm loc}}(\Om)$, then $Xu$
coincides with the regular distributional horizontal gradient of $u$. In general we have the following
gradient estimate \cite{TW} (see also \cite{HKM}).

\begin{proposition}[\cite{TW}]\label{gradient} Suppose u is $p$-superharmonic in $\Om$. Then $Xu$ belongs to $L^{r}_{{\rm loc}}(\Om)$,
where $r<\frac{M(p-1)}{M-1}$.
\end{proposition}
From  Proposition \ref{gradient} and the dominated convergence theorem we have
\begin{eqnarray*}
\int_{\Om}|Xu|^{p-2}Xu\cdot X\varphi dx=\lim_{k\ra\infty}\int_{\Om}|Xu_{k}|^{p-2}Xu_{k}\cdot X\varphi dx\geq 0
\end{eqnarray*}
whenever $\varphi\in C^{\infty}_{0}(\Om)$ and $\varphi\geq 0$, where $u_{k}=\min\{u, k\}$.
Thus the map
 $$\varphi\mapsto \int_{\Om}|Xu|^{p-2}Xu\cdot X\varphi dx$$
is a nonnegative distribution in $\Om$ for a $p$-superharmonic function  $u$. It follows that there is
a positive (not necessarily finite) Radon measure denoted by $\mu[u]$ such that
$$\int_{\Om}|Xu|^{p-2}Xu \cdot X\varphi dx=\int_{\Om}\varphi d\mu[u],\quad\quad \forall \varphi\in C_{0}^{\infty}(\Om),$$
or in short we write
\begin{eqnarray*}
-\Delta_{\mathbb G,\,p}\, u=\mu[u] \quad {\rm in} \quad \Om.
\end{eqnarray*}

The close relation between $p$-superharmonic functions and measures generated by them is established in the 
weak continuity theorem due to Trudinger and Wang \cite{TW}.

\begin{theorem}[\cite{TW}]\label{weakcont} Suppose that $\{u_{n}\}$ is a sequence of nonnegative
$p$-superharmonic functions in $\Om$ that converges a.e. to  a
$p$-superharmonic function $u$. Then
the sequence of corresponding measures $\{\mu[u_{n}]\}$ converges to $\mu[u]$ weakly, i.e.,
$$\lim_{n\rightarrow\infty}\int_{\Om}\varphi \, d\mu[u_{n}]=\int_{\Om}\varphi \, d\mu[u],$$
for all $\varphi\in C_{0}^{\infty}(\Om)$.
\end{theorem}

The following   pointwise estimates by means of Wolff's potentials were also proved in \cite{TW} which extend earlier results due to
Kilpel\"ainen and Mal\'y \cite{KM2} to the subelliptic setting. They will play an essential role in this paper.

\begin{theorem}[\cite{TW}]\label{potential}
Suppose  $u \ge 0$ is a $p$-superharmonic function in $B_{3r}(x)$.
If $\mu=-\Delta_{\mathbb G,\,p}\, u$, then
\begin{equation}\label{IKM}
C_{1} \, {\rm\bf W}_{1, \, p}^{r} \, \mu(x)\leq u(x)\leq C_{2}\, \inf_{B_{r}(x)} \, u + C_{3} \,
{\rm\bf W}_{1, \, p}^{2r} \, \mu(x),
\end{equation}
where $C_{1}, C_{2}$ and $C_{3}$ are positive constants depending only on $M$ and $p$. Consequently, if 
$-\Delta_{\mathbb G,\,p}\, u=\mu$ on $\mathbb G$ and $\inf_{\mathbb G}u =0$ then
\begin{equation}\label{IKM-G}
C_{1} \, {\rm\bf W}_{1, \, p}^{\infty} \, \mu(x)\leq u(x) \leq C_{3} \,
{\rm\bf W}_{1, \, p}^{\infty} \, \mu(x).
\end{equation}
\end{theorem}

\section{Lane-Emden type equations and related inequalities}\label{sec4}

In this section we fix a standard mollifier $\zeta$ on $\mathbb{G}$, i.e., a function $\zeta\in C_{0}^{\infty}(\mathbb G)$ which
is radially decreasing and is supported in $\{x\in\mathbb G: |x|\leq 1\}$ such that $\int\zeta dx =1$.
Also, for $n\geq 1$ we denote by $\zeta_{n}$ the function defined by $\zeta_n(x)=\frac{1}{n}\zeta(\frac{x}{n})$. The following theorem
gives an existence result and global pointwise estimates   for a quasilinear equation with measure data.

\begin{theorem}\label{EXIST}
Suppose that $\Om$ is bounded and $\mu$ is a nonnegative finite measure on $\Om$. Let $u_n$ be the unique solution in $S_{0}^{1,\,p}(\Om)$ of
\begin{equation}\label{equ}
-\Delta_{\mathbb G,\, p}\, u_n= \zeta_{n}*\mu \quad {\rm in}\quad \Om.
\end{equation}
Then there is a subsequence $\{u_{n_k}\}$ of $\{u_n\}$ and a $p$-superharmonic function $u$ on $\Om$ such that
$$u=\lim_{k\rightarrow \infty}u_{n_{k}} \quad {\rm a.e.}$$

Moreover, $u$ solves the equation 
\begin{eqnarray}\label{TRUN}
\left\{\begin{array}{rcl}
-\Delta_{\mathbb G,\, p}\, u&=&\mu \quad {\rm in~} \Om,\\
u&=&0 \quad {\rm on~} \partial \Om,
\end{array}
\right.
\end{eqnarray}
in the sense of Definition \ref{DEFEQ}. 
\end{theorem}

\begin{proof}
Let $\mu_n= \zeta_{n}*\mu$ and let $v_n$ be the unique solution in $S_{0}^{1,\,p}(B)$ of
\begin{equation}\label{eqv}
-\Delta_{\mathbb G,\, p}\, v_n=\mu_n \quad {\rm in}\quad B,
\end{equation}
where $B=B(a, 2R)$ with $R={\rm diam}(\Om)$ and $a\in\Om$ so that $\Om\subset B$.
We now extend $u_n$ by zero outside $\Om$ and set $\varphi=\min\{v_n-u_n,0\}$. Since $\varphi\in S_{0}^{1,\,p}(\Om)
\cap S_{0}^{1,\,p}(B)$ we can use it as a test function in \eqref{equ} and \eqref{eqv} to obtain
$$\int_{B}|Xv_n|^{p-2}Xv_n\cdot X\varphi dx-\int_{\Om}|Xu_n|^{p-2}Xu_n\cdot X\varphi dx =0,$$
or
$$\int_{B\cap\{v_n<u_n\}}|Xv_n|^{p-2}Xv_n\cdot X\varphi dx-\int_{B\cap\{v_n<u_n\}}|Xu_n|^{p-2}Xu_n \cdot X\varphi dx =0.$$

This gives
$$\int_{B\cap\{v_n<u_n\}}(|Xv_n|^{p-2}Xv_n -|Xu_n|^{p-2}Xu_n)\cdot X(v_n-u_n) dx =0.$$

Thus $\varphi=0$ a.e., or equivalently we have
\begin{equation}\label{unvn}
u_n\leq v_n \quad {\rm a.e.}
\end{equation}

Since $\mu_n(\Om)\leq \mu(\Om)$, for each $k>0$ we have the estimate
\begin{eqnarray}\label{GRAD}
\int_{\Om}|X(\min\{u_n,k\})|^p&=&\int_{\Om}|Xu_n|^{p-2}Xu_n\cdot X(\min\{u_n,k\}) \\
&=&\int_{\Om}\min\{u_n, k\}d\mu_n\leq k\mu_n(\Om)\nonumber\\
&\leq& k\mu(\Om).\nonumber
\end{eqnarray}
Consequently, by Sobolev's embedding theorem we obtain
$$\norm{\min\{u_n,k\}}_{L^{\frac{Mp}{M-p}}(\Om)}\leq C(k\mu(\Om))^{\frac{1}{p}}.$$

Hence
\begin{eqnarray*}
|\{u_n>k\}|&\leq& \Big(\frac{1}{k}\norm{\min\{u_n,k\}}_{L^{\frac{Mp}{M-p}}(\Om)}\Big)^{\frac{Mp}{M-p}}\\
&\leq& C \mu(\Om)^{\frac{M}{M-p}} k^{\frac{M(1-p)}{M-p}}.
\end{eqnarray*}

This gives
\begin{equation}\label{NES}
\norm{u_n}_{L^{p-1}(\Om)}\leq |\Om|^{\frac{p}{M(p-1)}}\norm{u_n}_{L^{\frac{M(p-1)}{M-p},\,\infty}}\leq
C |\Om|^{\frac{p}{M(p-1)}}\mu(\Om)^{\frac{1}{p-1}}.
\end{equation}

Now arguing as in \cite{KM1} we can find subsequences $\{u_{n_k}\}$, $\{v_{n_k}\}$ and $p$-superharmonic functions
$u$, $v$ on $\Om$ such that $u_{n_k}\rightarrow u$, $v_{n_k}\rightarrow v$ a.e. Hence from \eqref{GRAD}
and Theorem \ref{weakcont}  we see that $u$ is a distributional solution of \eqref{TRUN}. Similarly, $v$ also solves
\eqref{TRUN} in the distributional sense with $B$ in place of $\Om$, and \eqref{GRAD}, \eqref{NES} hold for $v_n$ with
$B$ in place of $\Om$ as well. In particular, this implies
\begin{equation}\label{NNES}
\norm{v}_{L^{p-1}(B)}\leq C R^{\frac{p}{(p-1)}}\mu(\Om)^{\frac{1}{p-1}}.
\end{equation}

Thus in view of \eqref{unvn} and Theorem \ref{potential} we get
\begin{eqnarray*}
u(x)&\leq& v(x)\leq C\,{\rm\bf W}_{1,\,p}^{\frac 2 3 d(x)}\mu(x) + C\inf_{B_{\frac 1 3d(x)}(x)}v\\
&\leq& C\,{\rm\bf W}_{1,\,p}^{2R}\mu(x) +
C d(x)^{\frac{-M}{p-1}}||v||_{L^{p-1}(B)}\\
&\leq& C\,{\rm\bf W}_{1,\,p}^{2R}\mu(x) +
C R^{\frac{-M}{p-1}}||v||_{L^{p-1}(B)},
\end{eqnarray*}
where $x\in \Om$ and $d(x)={\rm dist}(x, \partial B)$.  Note that  we have used the fact that
$d(x)\geq R$ in the last inequality. Finally, from this and \eqref{NNES} we obtain the pointwise estimate
$$u(x)\leq C\,\mathbf W_{1,\,p}^{2{\rm diam}(\Om)}(x)$$ 
for all $x\in \Om$. Thus $u$ solves \eqref{TRUN} in the potential theoretic sense and the proof is complete.
\end{proof}

\begin{corollary}\label{corexist} For any nonnegative finite measure $\mu$ on $\Om$, there exists a potential theoretic solution to equation
\eqref{eqme}.
\end{corollary}

We now construct a solution to a nonlinear equation with a power source term under a certain iterated Wolff's potential condition.
This condition  turns out to be  sharp as we will see later.

\begin{theorem}\label{SU2}
Let $\om$ be a nonnegative finite measure on $\Om$. Let $p>1 $ and $q>p-1$.
Suppose that $R={\rm diam}(\Om)$, and
\begin{equation}\label{CONDP}
{\rm\bf W}_{1,\,p}^{2R}\,({\rm\bf W}_{1,\,p}^{2R}\om)^{q}\leq C\,{\rm\bf W}_{1,\,p}^{2R}\om
\quad  {\rm a.e.,}
\end{equation}
where
\begin{equation}\label{condC}
C \leq \Big(\frac{q-p+1}{q A\max\{1,2^{p'-2}\}}\Big)^{q(p'-1)}\Big(\frac{p-1}{q-p+1}\Big),
\end{equation}
and $A$ is the constant in Definition \ref{DEFEQ}.
Then there is a solution $u\in L^{q}(\Om)$ to the equation
\begin{eqnarray}
\label{suff2}
\left\{\begin{array}{rcl}
-\Delta_{\mathbb G,\,p}\, u&=&u^{q}+\om \quad{\rm in~} \Om,\\
u&=&0 \quad {\rm on~} \partial\Om.
\end{array}
\right.
\end{eqnarray}
Moreover, for every $x$ in $\Om$,
$$ u(x)\leq \kappa \, {\rm\bf W}_{1,\,p}^{2R}\om(x),$$
where the constant $\kappa$ depends  only on $p, q, M$.
\end{theorem}

\begin{proof}
Let  $u^{(1)}_{n}$ solve the equation
\begin{eqnarray*}
\left\{\begin{array}{rcl}
-\Delta_{\mathbb G,\,p}\, u_n^{(1)}&=&\zeta_n*\om \quad{\rm in~}\Om,\\
u_n^{(1)}&=&0 \quad {\rm on~} \partial\Om.
\end{array}
\right.
\end{eqnarray*}
By Theorem \ref{EXIST}, there exists a function $u^{(1)}$ that satisfies
\begin{eqnarray}\label{U1}
\left\{\begin{array}{rcl}
-\Delta_{\mathbb G,\,p}\, u^{(1)}&=&\om \quad{\rm in~}\Om,\\
u^{(1)}&=&0 \quad {\rm on~} \partial\Om
\end{array}
\right.
\end{eqnarray}
in the sense of Definition \ref{DEFEQ}, and for a subsequence of $\{u^{(1)}_n\}$, still denoted by $\{u^{(1)}_n\}$, we have
\begin{equation}\label{LIM1}
u^{(1)}=\lim_{n\rightarrow\infty} u^{(1)}_{n} \quad {\rm a.e.}
\end{equation}

Similarly, let $u^{(2)}_n$ be a solution to the equation
\begin{eqnarray*}
\left\{\begin{array}{rcl}
-\Delta_{\mathbb G,\,p}\, u^{(2)}_n&=&\zeta_n*[u^{(1)}]^{q}+\zeta_n*\om \quad{\rm in~}\Om,\\
u^{(2)}_n&=&0 \quad {\rm on~} \partial\Om.
\end{array}
\right.
\end{eqnarray*}
Then by Theorem \ref{EXIST}, there exists a function $u^{(2)}$ that satisfies
\begin{eqnarray*}
\left\{\begin{array}{c}
-\Delta_{\mathbb G,\,p}\, u^{(2)}=[u^{(1)}]^{q} + \om \quad{\rm in~}\Om,\\
u^{(2)}=0 \quad {\rm on~} \partial\Om
\end{array}
\right.
\end{eqnarray*}
in the sense of Definition \ref{DEFEQ}, and for a subsequence of $\{u^{(2)}_n\}$, still denoted by $\{u^{(2)}_n\}$, we have
\begin{equation}\label{LIM2}
u^{(2)}=\lim_{n\rightarrow\infty} u^{(2)}_{n} \quad {\rm a.e.}
\end{equation}

As $u^{(1)}_n\leq u^{(2)}_n$ we see from \eqref{LIM1} and \eqref{LIM2} that
$u^{(1)}\leq u^{(2)}$ a.e. and hence everywhere since they are $p$-superharmonic.
Thus by induction we can find an
increasing sequence $\{u^{(k)}\}$ such that $u^{(1)}$ satisfies \eqref{U1} and for $k\geq 2$,
\begin{eqnarray}\label{EQUk}
\left\{\begin{array}{rcl}
-\Delta_{\mathbb G,\,p}\, u^{(k)}&=&[u^{(k-1)}]^{q} + \om \quad{\rm in~}\Om,\\
u^{(k)}&=&0 \quad {\rm on~} \partial\Om
\end{array}
\right.
\end{eqnarray}
in the sense of Definition \ref{DEFEQ}. Note that we have
$$u^{(1)}\leq A\, {\rm W}_{1,\,p}^{2R}\om, \qquad
u^{(k)}\leq A\, {\rm W}_{1,\,p}^{2R}([u^{(k-1)}]^{q}+\om)$$
for all $k\geq 2$. In view of  these estimates and the condition \eqref{CONDP} we get
\begin{eqnarray*}
u^{(2)}&\leq& A \max\{1,2^{p'-2}\}\Big[{\rm\bf W}^{2R}_{1,\,p}[u^{(1)}]^{q}
+{\rm\bf W}^{2R}_{1,\,p}\om\Big]\\
&\leq& A \max\{1,2^{p'-2}\}(c_{1}^{q(p'-1)}C+1){\rm\bf W}^{2R}_{1,\,p}\om\\
&=& c_{2}{\rm\bf W}^{2R}_{1,\,p}\om,
\end{eqnarray*}
where $c_1=A$ and $c_{2}=A \max\{1,2^{p'-2}\}(c_{1}^{q(p'-1)}C+1)$. By induction
we can find a sequence $\{c_{k}\}_{k\geq 1}$ of positive numbers
such that $$u^{(k)}\leq c_{k}{\rm\bf W}^{2R}_{1,\,p}\om,$$  with
$$c_{k}=A\max\{1,2^{p'-2}\}(c_{k-1}^{q(p'-1)}C+1)$$ for all $k\geq
2$. It is then easy to see that $c_{k}\leq \frac{A\max\{1,\,
2^{p'-2}\}q}{q-p+1}$ for all $k\geq 1$ as long as $C$ satisfies (\ref{condC}). Thus
$$u^{(k)}\leq \frac{A\max\{1,2^{p'-2}\}q}{q-p+1}{\rm\bf W}^{2R}_{1,\,p}\om.$$

Therefore, $\{u^{(k)}\}$ converges pointwise increasingly to a
nonnegative function $u$ for which
\begin{equation*}
u\leq \kappa\,{\rm W}_{1,\,p}^{2R}\om.
\end{equation*}

Finally, in view of \eqref{EQUk} and Theorem \ref{weakcont} we see that $u$ solves \eqref{suff2} in the
sense of Definition \ref{DEFEQ}. This completes the proof of the theorem.
\end{proof}

In the general context of homogeneous spaces, it was proved in \cite{SW} and \cite{Chr} 
that for $\lambda=8$, and for any (large negative) integer $m$, there are points $\{x_{j}^{k}\}\subset\mathbb G$
and a family of sets $\mathcal{D}_{m}=\{E^{k}_{j}\}$,  $k=m, m+1,\dots$ and $j=1,2,\dots$ such that

\begin{itemize}
\item[(i)] $B_{ \lambda^{k}}(x_{j}^{k})\subset E_{j}^{k}\subset B_{ \lambda^{k+1}}(x_{j}^{k}),$

\item[(ii)] For each fixed $k=m, m+1,\dots$, the sets $E^{k}_{j}$ are pairwise disjoint in $j$, and
$$\mathbb{G}=\bigcup_{j\geq  1}E_{j}^{k},$$

\item[(iii)] If $k<l$ {\rm ~then ~either~} $E^{k}_{j}\cap E^{l}_{i}=\emptyset$ {\rm ~or~} $E^{k}_{j}\subset E^{l}_{i}$.
\end{itemize}

We shall say that the family $\mathcal{D}=\bigcup_{m=-\infty}^{\infty}\mathcal{D}_{m}$ is a dyadic cube decomposition of
$\mathbb{G}$, and call sets in $\mathcal{D}$ dyadic cubes and denote them by $Q$. Note that the cubes in
$\mathcal{D}_{m_{1}}$ may have no relation to those in $\mathcal{D}_{m_{2}}$ if $m_{1}$ and $m_{2}$ are different. If
$Q=E^{k}_{j}\in \mathcal{D}_{m}$ for some $m$, we say $Q$ is centered at $x^{k}_{j}$ and define the side length of
$Q$ to be $\ell(Q)=\lambda^{k}$. We also denote by $Q^*$ the containing ball $B_{\lambda^{k+1}}(x_{j}^{k})$ of $Q$
and by $Q^{**}$ the ball $B_{ 2\lambda^{k+2}}(x_{j}^{k})$.

\begin{remark}\label{overlap}
Note  that if $Q^{**}=B_{2\lambda^{k+2}}(x^{k}_{j_{1}})$ and if
$\{E^{k}_{j_i}\}$, $i=1,\dots, d$, are the dyadic cubes that intersect $Q^{**}$ then obviously $Q^{**}\subset \cup_{i=1}^{m}E_{j_{i}}^{k}$.
Moreover, since each $E^{k}_{j_{i}}\subset B_{4\lambda^{k+2}}(x_{j_{1}}^{k})$ we obtain
$$c\, d \lambda^{k M}\leq \Big|\bigcup_{i=1}^{d} E^{k}_{j_i}\Big|\leq \left|B_{4\lambda^{k+2}}(x_{j_{1}}^{k})\right| \leq C\, \lambda^{k M},$$
which gives
$$d\leq C=C(\mathbb G).$$
This implies that the ball $Q^{**}=B_{2\lambda^{k+2}}(x^{k}_{j_{1}})$ is contained in the union of at most
$d$ dyadic cubes of side length $\lambda^{k}$ for some constant $d=d(\mathbb G)$.
\end{remark}

For an integer $m$, let  $\Lambda=\{\lambda_Q\}_{Q\in{\mathcal D}_{m}}$,
$\lambda_Q\geq 0$,  and let $\sigma$ be a positive locally finite Borel
measure on $\mathbb{G}$ such that $\lambda_Q = 0$ whenever $\sigma(Q)=0$. We will
 follow the  convention that $0 \cdot \infty = 0$. For $1<s<+\infty$, we define  \begin{equation*}\begin{split}  & A^{m}_1(\Lambda)=
\int_{\mathbb{G}}\Big[ \sum_{Q\in{\mathcal D}_{m}}
\frac{\lambda_Q}{\sigma(Q)}\chi_Q(x) \Big]^sd\sigma(x),\\   &
A_2^{m}(\Lambda)=\sum_{Q\in{\mathcal D}_{m}} \lambda_Q \Big[
\frac{1}{\sigma(Q)}\sum_{Q'\in {\mathcal D}_{m},\, Q'\subset Q}  \lambda_{Q'}\Big]^{s-1},\\
&A_3^{m}(\Lambda)=\int_{\mathbb{G}}\sup_{x\in Q\subset \mathcal{D}_{m}}\Big[
\frac1{\sigma(Q)}\sum_{Q'\in {\mathcal D}_{m},\, Q'\subset Q}\lambda_{Q'}\Big]^s  d\sigma(x).
\end{split}\end{equation*}

The  proof of  following proposition will be omitted as it is similar to the one  given in \cite{COV}
in the case $\mathbb{G}=\RR^N$ and $\mathcal{D}_{m}$ is the set of all standard dyadic cubes in $\RR^N$.

\begin{proposition}\label{proposition4.1}  Let
$\sigma$ be a positive locally finite Borel  measure on $\mathbb{G}$.  Let  $1<s<
+\infty$.  Then there exist constants $C_i >0, \, i = 1, 2, 3$, which depend
only  on $s$,  such that 
$$  A_1^{m} (\Lambda) \leq C_1 \, A_2^{m}(\Lambda) \leq C_2 \,
A_3^{m}(\Lambda)\leq C_3 \, A^{m}_1(\Lambda)$$
for all $\Lambda=\{\lambda_Q\}_{Q\in{\mathcal  D}_{m}}$ with $\lambda_Q\geq 0$, and $m\in\ZZ$.
\end{proposition}

We next consider the following quantities. For each integer $m$, a dyadic cube $P\in \mathcal{D}_{m}$, and
a nonnegative Borel measure $\mu$ on $\mathbb{G}$ we define
\begin{eqnarray*}
&(a)& B^{m}_{1}(P,\mu)=\sum_{Q\subset P} \Big[\frac{\mu(Q)}{\m{Q}^{1-
\frac{\alpha p}{Q}}}\Big]^{\frac{q}{p-1}}\m{Q},\\
&(b)& B^{m}_{2}(P,\mu)=\int_{P}\Big[\sum_{Q\subset P} \frac{\mu(Q)^{\frac{1}{p-1}}}
{\m{Q}^{(1-\frac{\alpha p}{Q})\frac{1}{p-1}}}\chi_{Q}(x)\Big]^{q}dx,\\
&(c)& B^{m}_{3}(P,\mu)=\int_{P}\Big[\sum_{Q\subset P}\frac{\mu(Q)}{\m{Q}^{1-\frac
{\alpha p}{Q}}}\chi_{Q}(x)\Big]^{\frac{q}{p-1}}dx.
\end{eqnarray*}
Here $\alpha>0$, $p>1$, $q>p-1$, and the sum is taken over all dyadic cubes $Q\in \mathcal{D}_{m}$ such that
$Q\subset P$.

\begin{proposition}\label{propB} There exist constants $C_{i}>0$, $i=1, 2, 3$, independent of $m$, $P$, and $\mu$ such that
\begin{equation}\label{chain}
B_{1}^{m}(P, \mu)\leq C_{1}B_{2}^{m}(P, \mu)\leq C_{2}B_{3}^{m}(P, \mu)\leq C_{3}B_{1}^{m}(P, \mu).
\end{equation}
\end{proposition}

\begin{proof} Let $\Lambda=\{\lambda_{Q}\}_{Q\in\mathcal{D}_{m}}$ where $\lambda_{Q}=\mu(Q)|Q|^{\frac{\alpha p}{M}}$
if $Q\subset P$ and $\lambda_{Q}=0$ otherwise. Applying Proposition \ref{proposition4.1} with $d\sigma=\chi_{P}dx$
and $s=\frac{q}{p-1}>1$ we obtain
\begin{eqnarray*}
B^{m}_{3}(P,\mu)&\leq& C \sum_{Q\subset P}\lambda_{Q}\Big[\frac{1}{|Q|}\sum_{Q'\subset Q}\lambda_{Q'}
\Big]^{\frac{q}{p-1}-1}\\
&=&C \sum_{Q\subset P}\mu(Q)|Q|^{\frac{\alpha p}{M}}\Big[\frac{1}{|Q|}\sum_{Q'\subset Q}\mu(Q')|Q'|^{\frac{\alpha p}{M}}
\Big]^{\frac{q}{p-1}-1}\\
&\leq& C \sum_{Q\subset P}\Big[\frac{\mu(Q)}{|Q|^{1-\frac{\alpha p}{M}}}\Big]^{\frac{q}{p-1}}|Q|=C\, B^{m}_{1}(P, \mu).
\end{eqnarray*}

Furthermore, since $\frac{q}{p-1}>1$,
\begin{eqnarray*}
B^{m}_{1}(P, \mu)&=&\int_{P}\sum_{Q\subset P}\Big[\frac{\mu(Q)}{|Q|^{1-\frac{\alpha p}{M}}}\Big]^{\frac{q}{p-1}}
\chi_{Q}(x)dx\leq B^{m}_{3}(P,\mu)\\
&\leq& C \int_{P}\sup_{x\in Q\subset P}\Big[\frac{1}{|Q|}\sum_{Q'\subset Q}\lambda_{Q'}\Big]^{\frac{q}{p-1}}dx\\
&\leq& C \int_{P}\sup_{x\in Q\subset P}\Big[\frac{\mu(Q)}{|Q|^{1-\frac{\alpha p}{M}}}\Big]^{\frac{q}{p-1}}dx
\leq C\, B^{m}_{2}(P, \mu),
\end{eqnarray*}
where we have used Proposition \ref{proposition4.1} in the second inequality. We next observe that for
$p\leq 2$, $B^{m}_{2}(P,\mu)\leq B^{m}_{3}(P,\mu)$. Thus it remains to show that for $p>2$,
$B^{m}_{2}(P,\mu)\leq C\, B^{m}_{3}(P,\mu)$. Since $q>p-1>1$, by Proposition \ref{proposition4.1} we have

\begin{eqnarray*}
B^{m}_{2}(P,\mu)&=&\int_{P}\Big[\sum_{Q\subset P} \frac{\mu(Q)^{\frac
{1}{p-1}}}{\m{Q}^{(1-\frac{\alpha p}{M})\frac{1}{p-1}}}\chi_{Q}(x)\Big]^{q}dx\\
&\leq& C\sum_{Q\subset P}\frac{\mu(Q)^\frac{1}{p-1}}{\m{Q}^{(1-\frac{\alpha p}
{M})\frac{1}{p-1}+q-2}} \Big[\sum_{Q'\subset Q} \frac{\mu(Q')^{\frac{1}{p-1}}}
{\m{Q'}^{(1-\frac{\alpha p}{M})
\frac{1}{p-1}-1}}\Big]^{q-1}.\nonumber
\end{eqnarray*}

On the other hand, by H\"older's inequality the sum in the above square brackets can be estimated by
\begin{eqnarray*}
\lefteqn{\sum_{Q'\subset Q}\Big(\mu(Q')^\frac{1}{p-1}\m{Q'}^\epsilon\Big)\m{Q'}^
{-(1-\frac{\alpha p}{M})\frac{1}{p-1}+1-\ep}}\\
&\leq& \Big(\sum_{Q'\subset Q}\mu(Q')^\frac{r'}{p-1}\m{Q'}^{\ep r'}\Big)^{\frac
{1}{r'}}\Big(\sum_{Q'\subset Q}\m{Q'}^{-r(1-\frac{\alpha p}{M})\frac{1}{p-1}+r-r\ep}\Big)^
\frac{1}{r},
\end{eqnarray*}
where $r'=p-1>1$, $r=\frac{p-1}{p-2}$ and $\ep>0$ is chosen so
that $-r(1-\frac{\alpha p}{M})\frac{1}{p-1}+r-r\ep>1$, i.e., $0<\ep<\frac
{\alpha p}{(p-1)M}$. Therefore,
\begin{eqnarray*}
\sum_{Q'\subset Q} \frac{\mu(Q')^{\frac{1}{p-1}}}{\m{Q'}^{(1-\frac{\alpha p}{M})
\frac{1}{p-1}-1}}&\leq& C\mu(Q)^{\frac{1}{p-1}}\m{Q}^{\ep}\m{Q}^{-(1-\frac{
\alpha p}{M})\frac{1}{p-1}+1-\ep}\\
&=& C\frac{\mu(Q)^\frac{1}{p-1}}{\m{Q}^{(1-\frac{\alpha p}{M})\frac{1}{p-1}-1}}.
\end{eqnarray*}
Hence, combining the preceding inequalities, we obtain
\begin{eqnarray*}
B^{m}_{2}(P,\mu)&\leq& C\sum_{Q\subset P}\frac{\mu(Q)^\frac{1}{p-1}}{\m{Q}^{(1-
\frac{\alpha p}{M})
\frac{1}{p-1}+q-2}} \Big[\frac{\mu(Q)^{\frac{1}{p-1}}}{\m{Q}^{(1-\frac{
\alpha p}{M})\frac{1}{p-1}-1}}\Big]^{q-1}\\
&=&C\sum_{Q\subset P} \frac{\mu(Q)^{\frac{q}{p-1}}}{\m{Q}^{(1-\frac{\alpha p}
{M})\frac{q}{p-1}-1}}=C\, B^{m}_{1}(P,\mu)\leq C\,B^{m}_{3}(P,\mu).
\end{eqnarray*}
This completes the proof of the proposition.
\end{proof}

\begin{remark}\label{Q**} From Remark \ref{overlap} we see that
\begin{eqnarray*}
\sum_{Q'\subset Q}\mu(Q'^{**})|Q'|^{\beta}&=&\sum_{k=0}^{\infty}\sum_{\substack{\ell(Q')=\ell(Q)/\lambda^{k},\\
Q'\subset Q}}
\mu(Q'^{**})|Q'|^{\beta}\\
&\leq& C \sum_{k=0}^{\infty} \lambda^{-k\beta M}\ell(Q)^{\beta M}\sum_{\substack{\ell(Q')=\ell(Q)/\lambda^{k},\\
Q'\subset Q}}\mu(Q'^{**}) \\
&\leq& C \mu(Q^{**})|Q|^{\beta}
\end{eqnarray*}
for any $\beta >0$. Thus the chain of inequalities in \eqref{chain} still holds if $\mu(Q)$ is
replaced by $\mu(Q^{**})$ in the definition of $B^{m}_{i}(P,\mu)$, $i=1, 2, 3$.
\end{remark}

\begin{lemma} \label{DZE} Let $\alpha>0$ and $p>1$. Then for any integer $m$,
\begin{equation}\label{loweres}
\mathbf{W}^{r}_{\alpha,\,p}\mu(x)\geq c \sum_{\substack{Q\in\mathcal{D}_{m},\\\ell(Q)\leq\lambda^{-3}r}}
\Big[\frac{\mu(Q)}{|Q|^{1-\frac{\alpha p}{M}}}\Big]^{\frac{1}{p-1}}\chi_{Q}(x),
\end{equation}
and
\begin{equation}\label{upperes}
\int_{\lambda^{m}r}^{r} \Big[\frac{\mu(B_{t}(x))}{t^{M-\alpha p}}\Big]^{\frac{1}{p-1}} \frac{dt}{t}\leq C
\sum_{\substack{Q\in\mathcal{D}_{m+[\log_{\lambda}r]},\\\ell(Q)\leq r}}
\Big[\frac{\mu(Q^{**})}{|Q|^{1-\frac{\alpha p}{M}}}\Big]^{\frac{1}{p-1}}\chi_{Q}(x).
\end{equation}
In \eqref{upperes} $[\log_{\lambda}r]$ stands for the integral part of the real number $\log_{\lambda}r$.

\end{lemma}
\begin{proof}
To prove \eqref{loweres} we may assume that $\lambda^{m}\leq\lambda^{-3}r$ since dyadic cubes in
$\mathcal{D}_{m}$ have side length not smaller than $\lambda^{m}$. Then $[\log_{\lambda}r]-3\geq m$. Observe that
\begin{eqnarray*}
&&\mathbf{W}_{\alpha,\, p}^{r}\mu(x)= \sum_{k=0}^{\infty}\int_{\lambda^{-k-1}r}^{\lambda^{-k}r} \Big[\frac{
\mu(B_{t}(x))}{t^{M-\alpha p}}\Big]^{\frac{1}{p-1}}\frac{dt}{t}\\
&&\geq c \sum_{k=0}^{\infty} \Big[\frac{
\mu(B_{\lambda^{-k-1}r}(x))}{(\lambda^{-k-1}r)^{M-\alpha p}}\Big]^{\frac{1}{p-1}}\\
&&\geq c \sum_{k=0}^{\infty}\sum_{E_{j}^{-k-3+[\log_{\lambda} r]}\in\mathcal{D}_{m}}\Big[\frac{\mu(E_{j}^{-k-3+[\log_{\lambda} r]})}{\ell(E_{j}^{-k-3+[\log_{\lambda} r]})^{M-\alpha p}}\Big]
\chi_{E_{j}^{-k-3+[\log_{\lambda} r]}}(x),
\end{eqnarray*}
where the last inequality follows from the fact that for $k\geq 0$ and $x\in E_{j}^{-k-3+[\log_{\lambda} r]}\in\mathcal{D}_{m}$,
$$E_{j}^{-k-3+[\log_{\lambda} r]}\subset B_{\lambda^{-k-2+[\log_{\lambda}r]}}(x_{j}^{-k-3+[\log_{\lambda}r]})\subset
B_{\lambda^{-k-1}r}(x).$$
Thus we obtain \eqref{loweres}. Similarly, to prove \eqref{upperes} we may assume that $m< 0$ and we have
\begin{eqnarray*}
&&\int_{\lambda^{m}r}^{r} \Big[\frac{\mu(B_{t}(x))}{t^{M-\alpha p}}\Big]^{\frac{1}{p-1}} \frac{dt}{t}
= \sum_{k=0}^{|m|-1}\int_{\lambda^{-k-1}r}^{\lambda^{-k}r} \Big[\frac{
\mu(B_{t}(x))}{t^{M-\alpha p}}\Big]^{\frac{1}{p-1}}\frac{dt}{t}\\
&&\leq C \sum_{k=0}^{|m|-1} \Big[\frac{
\mu(B_{\lambda^{-k}r}(x))}{(\lambda^{-k}r)^{M-\alpha p}}\Big]^{\frac{1}{p-1}}\\
&&\leq C \sum_{k=0}^{|m|-1}\sum_{j} \Big[\frac{
\mu(B_{2\lambda^{-k}r}(x_{j}^{-k-1+[\log_{\lambda}r]}))}{(\lambda^{-k}r)^{M-\alpha p}}\Big]^{\frac{1}{p-1}}\chi_{E_{j}^{-k-1+[\log_{\lambda}r]}}(x).
\end{eqnarray*}
Here $E_{j}^{-k-1+[\log_{\lambda}r]}\in D_{m+[\log_{\lambda}r]}$, and the last inequality follows since
for $x\in E_{j}^{-k-1+[\log_{\lambda}r]}$, we have
$x\in B_{\lambda^{-k+[\log_{\lambda}r]}}(x_{j}^{-k-1+[\log_{\lambda}r]})$ and hence
$$B_{\lambda^{-k}r}(x)\subset B_{2\lambda^{-k}r}(x_{j}^{-k-1+[\log_{\lambda}r]}).$$
This gives (\ref{upperes}) and completes the proof of the lemma.
\end{proof}

The result obtained in the  following theorem may be considered as an analogue of Wolff's  inequality (see \cite{HW}, \cite{PV1})
which is crucial in our approach to quasilinear Lane-Emden type equations later on.

\begin{theorem}\label{cont-discrete} Let $\alpha>0$, $p>1$ and $q>p-1$. Then for any $0<r<\infty$ and any nonnegative Borel measure $\mu$ on $\mathbb{G}$,
\begin{eqnarray*}
\int_{\mathbb G}\left[\mathbf{W}_{\alpha,\, p}^{r}\mu(x)\right]^{q}dx
\cong\sup_{m\in\ZZ}\, \sum_{Q\in\mathcal{D}_{m},\, \ell(Q)\leq r}
\Big[\frac{\mu(Q)}{|Q|^{1-\frac{\alpha p}{M}}}\Big]^{\frac{q}{p-1}}|Q|,
\end{eqnarray*}
where the constants of equivalence are independent of $r$ and $\mu$.
\end{theorem}
\begin{proof} Let $k\in\ZZ$ be such that $\frac{r}{\lambda}<\lambda^{k}\leq r$. For any interger $m\leq0$,
by Lemma \ref{DZE} we have
\begin{eqnarray*}
\lefteqn{\int_{\mathbb G}\Big\{\int_{\lambda^{m}r}^{r} \Big[\frac{\mu(B_{t}(x))}{t^{M-\alpha p}}\Big]^
{\frac{1}{p-1}} \frac{dt}{t}\Big\}^{q}dx}\\
&=&\sum_{j:\, E^{k}_{j}\in\mathcal{D}_{m+[\log_{\lambda}r]}}\int_{E^{k}_{j}}  \Big\{\int_{\lambda^{m}r}^{r}
\Big[\frac{\mu(B_{t}(x))}{t^{M-\alpha p}}\Big]^{\frac{1}{p-1}} \frac{dt}{t} \Big\}^{q}dx\\
&\leq& C \sum_{j:\, E^{k}_{j}\in\mathcal{D}_{m+[\log_{\lambda}r]}}\int_{E^{k}_{j}}\Big\{\sum_{\substack{Q\in\mathcal{D}_
{m+[\log_{\lambda}r]},\\\ell(Q)\leq r}}
\Big[\frac{\mu(Q^{**})}{|Q|^{1-\frac{\alpha p}{M}}}\Big]^{\frac{1}{p-1}}\chi_{Q}(x)\Big\}^{q}dx\\
&=& C \sum_{j:\, E^{k}_{j}\in\mathcal{D}_{m+[\log_{\lambda}r]}}\int_{E^{k}_{j}}\Big\{\sum_{
\substack{Q\in\mathcal{D}_{m+[\log_{\lambda}r]},\\Q\subset E^{k}_{j}}}
\Big[\frac{\mu(Q^{**})}{|Q|^{1-\frac{\alpha p}{M}}}\Big]^{\frac{1}{p-1}}\chi_{Q}(x)\Big\}^{q}dx.
\end{eqnarray*}
Thus by Proposition \ref{propB} and Remark \ref{Q**} we obtain
\begin{eqnarray*}
\lefteqn{\int_{\mathbb G}\Big\{\int_{\lambda^{m}r}^{r} \Big[\frac{\mu(B_{t}(x))}
{t^{M-\alpha p}}\Big]^{\frac{1}{p-1}} \frac{dt}{t}\Big\}^{q}dx}\\
&\leq& C \sum_{j:\, E^{k}_{j}\in\mathcal{D}_{m+[\log_{\lambda}r]}}\sum_{\substack{Q\in\mathcal{D}_{m+[\log_{\lambda}r]},
\\Q\subset E^{k}_{j}}}
\Big[\frac{\mu(Q^{**})}{|Q|^{1-\frac{\alpha p}{M}}}\Big]^{\frac{q}{p-1}}|Q|\\
&=& C \sum_{\substack{Q\in\mathcal{D}_{m+[\log_{\lambda}r]},\\\ell(Q)\leq r}}
\Big[\frac{\mu(Q^{**})}{|Q|^{1-\frac{\alpha p}{M}}}\Big]^{\frac{q}{p-1}}|Q|\\
&\leq& C \sum_{\substack{Q\in\mathcal{D}_{m+[\log_{\lambda}r]},\\\ell(Q)\leq r}}
\Big[\frac{\mu(Q)}{|Q|^{1-\frac{\alpha p}{M}}}\Big]^{\frac{q}{p-1}}|Q|,
\end{eqnarray*}
where the last inequality follows from Remark \ref{overlap}. This gives
\begin{eqnarray*}
\int_{\mathbb G}\left[\mathbf{W}_{\alpha,\, p}^{r}\mu(x)\right]^{q}dx
&\leq& C \sup_{m\in\ZZ} \sum_{\substack{Q\in\mathcal{D}_{m+[\log_{\lambda}r]},\\\ell(Q)\leq r}}
\Big[\frac{\mu(Q)}{|Q|^{1-\frac{\alpha p}{M}}}\Big]^{\frac{q}{p-1}}|Q|\\
&\leq& C \sup_{m\in\ZZ}\, \sum_{Q\in\mathcal{D}_{m},\, \ell(Q)\leq r}
\Big[\frac{\mu(Q)}{|Q|^{1-\frac{\alpha p}{M}}}\Big]^{\frac{q}{p-1}}|Q|.
\end{eqnarray*}

Analogously, from \eqref{loweres} in Lemma \ref{DZE} we obtain
$$\int_{\mathbb G}\left[\mathbf{W}_{\alpha,\, p}^{\lambda^{3}r}\mu(x)\right]^{q}dx
\geq C \sup_{m\in\ZZ}\, \sum_{Q\in\mathcal{D}_{m},\, \ell(Q)\leq r}
\Big[\frac{\mu(Q)}{|Q|^{1-\frac{\alpha p}{M}}}\Big]^{\frac{q}{p-1}}|Q|.$$
Note that
$$\int_{\mathbb G}\left[\mathbf{W}_{\alpha,\, p}^{\lambda^{3}r}\mu(x)\right]^{q}dx \leq C\,\int_{\mathbb G}
\left[\mathbf{W}_{\alpha,\, p}^{r}\mu(x)\right]^{q}dx$$
if we can show that
\begin{equation}\label{5.4}
\int_{\mathbb G}\left[\frac{\mu(B_{\lambda^{3} r}(x))}{r^{M-\alpha p}}\right]^{\frac{q}{p-1}}dx \leq C\,\int_{\mathbb G}
\left[\mathbf{W}_{\alpha,\, p}^{r}\mu(x)\right]^{q}dx.
\end{equation}

To prove \eqref{5.4}, we choose an interger $k$ so that $\lambda^{k+1}< \frac{r}{4}\leq
\lambda^{k+2}$ and as in Remark \ref{overlap}, it can be  seen that
for $x\in E^{k}_{j_1}\subset \mathcal{D}_{k}$ for some $j_1\geq 1$ the ball $B_{\lambda^3 r}(x)$ is contained in the union of at most
$d$ cubes in $\{E^{k}_{j}\}_{j\geq 1}\subset \mathcal{D}_{k}$ for some constant $d=d(\mathbb G)$. That is,
$$B_{\lambda^3 r}(x) \subset\bigcup_{i=1}^{d} E^{k}_{j_{i}}. $$

Thus  we obtain
\begin{eqnarray*}
\int_{\mathbb G}\mu(B_{\lambda^{3} r}(x))^{\frac{q}{p-1}}dx &=& \sum_{j}\int_{E^{k}_{j}} \mu(B_{\lambda^{3}r}(x))^
{\frac{q}{p-1}}dx\\
&\leq&C \sum_{j}\int_{E^{k}_{j}} \left[\mu(E^k_{j_{1}})^
{\frac{q}{p-1}}+ \dots +\mu(E^k_{j_{d}})^{\frac{q}{p-1}}  \right]dx\\
&\leq&C \sum_{j} \left[\int_{E^{k}_{j_1}} \mu(E^k_{j_{1}})^
{\frac{q}{p-1}}+ \dots +\int_{E^{k}_{j_d}}\mu(E^k_{j_{d}})^{\frac{q}{p-1}}  \right]dx\\
&\leq& C \sum_{j}\int_{E^{k}_{j}}\mu(E^{k}_{j})^{\frac{q}{p-1}}dx.
\end{eqnarray*}

Therefore, we get
\begin{eqnarray*}
\int_{\mathbb G}\Big[\frac{\mu(B_{\lambda^{3}r}(x))}{r^{M-\alpha p}}\Big]^{\frac{q}{p-1}}dx
&\leq& C\sum_{j}\int_{E^{k}_{j}}\Big[\frac{\mu(B_{\frac{r}{2}}(x))}{r^{M-\alpha p}}\Big]^
{\frac{q}{p-1}}dx\\
&\leq& C\sum_{j}\int_{E^{k}_{j}}\Big\{\int_{0}^{r}\Big[\frac{\mu(B_{t}(x))}{t^{M-\alpha p}}\Big]^{\frac{1}{p-1}}
\frac{dt}{t}\Big\}^{q}dx,
\end{eqnarray*}
which gives \eqref{5.4} and completes the proof of the theorem.
\end{proof}

We also have a continuous version of Wolff's inequality which is known  in the standard Euclidean setting  \cite{PV1}.

\begin{theorem}\label{W-G}  Let $\alpha>0$, $p>1$ and $q>p-1$. Then for any $0<r<\infty$ and any
nonnegative Borel measure $\mu$ on $\mathbb{G}$,
$$\norm{{\rm\bf W}_{\alpha,\,p}^{r}\mu}_{L^{q}(dx)}^{q}\cong \norm{{\rm\bf W}_{\alpha p,\,
\frac{q}{q-p+1}}^{r}\mu}_{L^{1}(d\mu)}\cong \norm{{\rm\bf G}_{\alpha p}\mu}_{L^{\frac{q}{p-1}}(dx)}^{\frac{q}{p-1}},$$
where the constants in these equivalences are independent $\mu$.
\end{theorem}

\begin{proof} By arguing as in the proof of Theorem \ref{cont-discrete} we also find that
$$\norm{{\rm\bf W}_{\alpha p,\,
\frac{q}{q-p+1}}^{r}\mu}_{L^{1}(d\mu)}\cong \sup_{m\in\ZZ}\, \sum_{Q\in\mathcal{D}_{m},\, \ell(Q)\leq r}
\Big[\frac{\mu(Q)}{|Q|^{1-\frac{\alpha p}{M}}}\Big]^{\frac{q}{p-1}}|Q|.$$
Thus by Theorem \ref{cont-discrete},
$$\norm{{\rm\bf W}_{\alpha,\,p}^{r}\mu}_{L^{q}(dx)}^{q}\cong \norm{{\rm\bf W}_{\alpha p,\,
\frac{q}{q-p+1}}^{r}\mu}_{L^{1}(d\mu)}.$$
On the other hand, by Wolff's inequality (see \cite{CLL}, and 
 \cite{AH}, \cite{Tu} in the Euclidean setting),
$$\norm{{\rm\bf W}_{\alpha p,\,
\frac{q}{q-p+1}}^{r}\mu}_{L^{1}(d\mu)}\cong \norm{{\rm\bf G}_{\alpha p}\mu}_{L^{\frac{q}{p-1}}(dx)}^{\frac{q}{p-1}},$$
which gives the theorem.
\end{proof}

Similarly, in the case $r=\infty$ we have the following Wolff type inequality.
\begin{theorem}\label{W-G-R}  Let $\alpha>0$, $1<p<M/\alpha$ and $q>p-1$. Then for any nonnegative Borel measure $\mu$ on $\mathbb{G}$,
$$\norm{{\rm\bf W}_{\alpha,\,p}^{\infty}\mu}_{L^{q}(dx)}^{q}\cong \norm{{\rm\bf W}_{\alpha p,\,
\frac{q}{q-p+1}}^{\infty}\mu}_{L^{1}(d\mu)}\cong \norm{{\rm\bf I}_{\alpha p}\mu}_{L^{\frac{q}{p-1}}(dx)}^{\frac{q}{p-1}},$$
where the constants in these equivalences are independent 
of $\mu$.
\end{theorem}

We are now in a position to prove the first main result of the paper.

\begin{proof}[{\bf Proof of Theorem \ref{Dmain1}}]
It is known that $({\rm i})\Leftrightarrow ({\rm ii})$ at least in the elliptic case, i.e., on $\RR^N$
(see, e.g., \cite{AH}), and  the proof given in \cite{AH} works also on Carnot groups. Next,
by duality and Theorem \ref{W-G} we have
$({\rm i})\Leftrightarrow ({\rm iii})$. Also, observe that $({\rm iii})\Rightarrow ({\rm iv})$ by letting $g=\chi_{B}$ in
\eqref{DMA2}. Moreover, we have the  implication ${\rm (iv)}\Rightarrow {\rm (v)}$ by following the proof given in
\cite[Theorem 2.10]{PV1} in the elliptic case. Thus from  Theorem \ref{SU2} we obtain the last conclusion of the theorem.

Therefore, it is left to show that the existence of a solution $u$ to \eqref{DMA1} implies (i). To this end, we let
$\mu=u^q+\om$ and $\delta(x)={\rm dist}(x,\partial\Om)$.
From the lower Wolff's potential estimate in Theorem \ref{potential} we have
$$u(x)\geq C\,\mathbf{W}^{\frac{\delta(x)}{3}}_{1,\, p}\mu(x),\qquad\forall x\in\Om.$$

By Lemma \ref{DZE} we obtain for any $m\in\ZZ$,
$$\Big\{\sum_{\substack{Q\in\mathcal{D}_{m}, \\ \ell(Q)\leq\lambda^{-3}\frac{\delta(x)}{3}}}
\Big[\frac{\mu(Q)}{|Q|^{1-\frac{p}{M}}}\Big]^{\frac{1}{p-1}}\chi_{Q}(x)\Big \}^{q} \chi_{\Om}(x)dx\leq C\, d\mu,$$
and thus
\begin{eqnarray*}
&&\int_{\Om}\Big \{\sum_{\substack{Q\in\mathcal{D}_{m},\\ \ell(Q)\leq\lambda^{-3}\frac{\delta(x)}{3}}}
\Big[\frac{\mu(Q)}{|Q|^{1-\frac{p}{M}}}\Big]^{\frac{1}{p-1}}\chi_{Q}(x)\Big \}^{q}
\Big({\rm\bf M}^{{\rm dy},\, \mathcal{D}_{m}}_{\mu}g\Big)^{\frac{q}{p-1}}dx\\
&&\leq C \int_{\mathbb G}\Big({\rm\bf M}^{{\rm dy},\, \mathcal{D}_{m}}_{\mu}g\Big)^{\frac{q}{p-1}}d\mu,
\end{eqnarray*}
which holds for all $g\in L^{\frac{q}{p-1}}_{\mu}$.
Here ${\rm\bf M}^{{\rm dy},\, \mathcal{D}_{m}}_{\mu}$ denotes the dyadic Hardy-Littlewood maximal function
defined for a locally $\mu$-integrable function $f$ by
$${\rm\bf M}^{{\rm dy},\, \mathcal{D}_{m}}_{\mu}f(x)=\sup_{x\in Q\in \mathcal{D}_{m}}\frac{\int_{Q}|f|d\mu}{\mu(Q)}.$$

Next, since ${\rm\bf M}^{{\rm dy},\, \mathcal{D}_{m}}_{\mu}$ is bounded on $L^{s}_{\mu}$, $s>1$,  we get
\begin{equation}\label{dix}
\int_{\Om}\Big\{\sum_{\substack{Q\in\mathcal{D}_{m},\\\ell(Q)\leq\lambda^{-3}\frac{\delta(x)}{3}}}
\Big[\frac{\int_{Q}g d\mu}{|Q|^{1-\frac{p}{M}}}\Big]^{\frac{1}{p-1}}\chi_{Q}(x)\Big \}^{q}dx
\leq C \int_{\mathbb{G}}g^{\frac{q}{p-1}}d\mu
\end{equation}
for all $g\in L^{\frac{q}{p-1}}_{\mu}$, $g\geq 0$. We now  let
$r_{0}={\rm dist}({\rm supp}(\om),\partial\Omega)$ and $\Om'=\{x\in\Om: {\rm dist}(x, {\rm supp(\om)})<r_0/2\}$, where the distance is taken with respect to the Carnot-Carath\'eodory metric. Note that for $x\in \Om$ with $\delta(x)\leq r_0/4$ we have $Q\cap\Om'=\emptyset$ whenever $x\in Q\in \mathcal{D}_{m}$ and 
$\ell(Q)\leq\lambda^{-3}\frac{\delta(x)}{3}$. Inequality \eqref{dix} then implies that for all
 $g\in L_{\mu}^{\frac{q}{p-1}}$, $g\geq 0$  we have
\begin{equation*}
\int_{\mathbb{G}}\Big\{\sum_{\substack{Q\in\mathcal{D}_{m},\\\ell(Q)
\leq\lambda^{-3}\frac{r_{0}}{12}}}
\Big[\frac{\int_{Q}g d\mu}{|Q|^{1-\frac{p}{M}}}\Big]^{\frac{1}{p-1}}\chi_{Q}(x)\Big\}^{q}dx
\leq C \int_{\mathbb{G}}g^{\frac{q}{p-1}}d\mu
\end{equation*}
provided  ${\rm supp}(g)\subset\Om'$. 
Thus for $k\in\ZZ$ such that $\lambda^{-4}\frac{r_{0}}{12}\leq \lambda^{k}\leq \lambda^{-3}\frac{r_{0}}{12}$ we find
\begin{equation*}
\sum_{j:\, E^{k}_{j}\in \mathcal{D}_{m}}\int_{E^{k}_{j}}\Big\{\sum_{\substack{Q\in\mathcal{D}_{m},\\Q\subset
E^{k}_{j}}}
\Big[\frac{\int_{Q}g d\mu}{|Q|^{1-\frac{p}{M}}}\Big]^{\frac{1}{p-1}}\chi_{Q}(x)\Big\}^{q}dx
\leq C \int_{\mathbb{G}}g^{\frac{q}{p-1}}d\mu,
\end{equation*}
and hence by Proposition  \ref{propB},
\begin{equation}\label{dismu-mu}
\sum_{\substack{Q\in\mathcal{D}_{m},\\\ell(Q)\leq \lambda^{-3}\frac{r_{0}}{12}}}
\Big[\frac{\int_{Q}gd\mu}{|Q|^{1-\frac{p}{M}}}\Big]^{\frac{q}{p-1}}|Q|\leq C \int_{\mathbb{G}}g^{\frac{q}{p-1}}d\mu 
\end{equation}
provided ${\rm supp}(g)\subset \Om'$. As \eqref{dismu-mu} holds for all $m\in \mathbb Z$, it  follows from Theorems \ref{cont-discrete} and \ref{W-G} that
\begin{equation}\label{mu-mu}
\int_{\mathbb{G}}\left[\mathbf{G}_{p}(gd\mu)\right]^{\frac{q}{p-1}}dx \leq C \int_{\mathbb{G}}g^{\frac{q}{p-1}}d\mu
\end{equation}
for all $g\in L^{\frac{q}{p-1}}(d \mu)$, $g\geq 0$, and ${\rm supp}(g)\subset \Om'$. Inequality \eqref{mu-mu},  duality, and the facts that
$\om\leq \mu$ and ${\rm supp}(\om)\subset\Om'$ finally yield
$$\int_{\mathbb{G}}\mathbf{G}_{p}(f)^{\frac{q}{q-p+1}}d\om \leq C \int_{\mathbb{G}}f^{\frac{q}{q-p+1}}dx$$
for all $f\in L^{\frac{q}{q-p+1}}$, $f\geq 0$. This completes the proof of the theorem.
\end{proof}
 We next prove Theorem \ref{removeforp}.

\begin{proof}[{\bf Proof of Theorem \ref{removeforp}}] We first  suppose that  ${C}_{p, \, \frac{q}{q-p+1}}(E)=0$.
Since $\frac{pq}{(q-p+1)}>p$ by Theorems 4.1 and 4.9 in 
\cite{Lu} we find ${C}_{1, \, p}(E)=0$. On the other hand, by a result in \cite{Fol} we have the identification
$$S^{1,\, p}(\mathbb G)={\rm\bf G}_{1}(L^{p}(\mathbb G))$$ 
with  $\norm{u}_{S^{1,\, p}(\mathbb G)}\cong \norm{f}_{L^{p}(\mathbb G)}$
for any $u$ with $u={\rm\bf G}_{1}(f)$. Thus we also have 
\begin{equation}\label{capOm}
{\rm cap}_{1,\, p}(E, \Om)=0,
\end{equation}
where ${\rm cap}_{1,\, p}(\cdot, \Om)$ is a relative capacity adapted to $\Om$ (see \cite{TW}, \cite{HKM}) defined by
$${\rm cap}_{1,\, p}(E, \Om)=\inf\left\{ \int_{\Om} |X\varphi|^p dx: \varphi\in C^{\infty}_{0}(\Om), \varphi\geq \chi_E \right\}.$$

Let $u$ be a solution of (\ref{omE}). Using \eqref{capOm} and adapting the argument in \cite{HKM} to this setting we see that the
function
\begin{equation}
\tilde{u}(x):=\left\{\begin{array}{c}
u(x),\quad x\in\Om\setminus E,\\
\displaystyle{\liminf_{\substack{y\rightarrow x,\,y\not\in E}}} \, u(y),\quad x\in E
\end{array}
\right.
\end{equation}
is a $p$-superharmonic extension of $u$ to the whole $\Omega$. We now let $\varphi$ be an arbitrary nonnegative function in $C_{0}^{\infty}(\Om)$.
As in \cite[Lemme 2.2]{BP}, we can construct a sequence $\{\varphi_{n}\}$ of nonnegative
functions in $C_{0}^{\infty}(\Om\setminus E)$ such that
\begin{equation}0\leq \varphi_{n}\leq \varphi; \qquad
\varphi_{n}\rightarrow\varphi \quad {\rm almost ~everywhere}.
\end{equation}

By Fatou's lemma we have
\begin{eqnarray*}
\int_{\Om} \tilde{u}^{q} \, \varphi \, dx &\leq& \liminf_{n\rightarrow\infty}\int_{\Om} \tilde{u}^{q} \,
\varphi_{n} \, dx = \liminf_{n\rightarrow\infty}\int_{\Om}\varphi_{n} \, d\mu[\tilde{u}]\\
&\leq& \int_{\Om} \, \varphi \, d\mu[\tilde{u}] <+\infty.
\end{eqnarray*}
Here $\mu[\tilde{u}]$ is the measure generated by $\tilde{u}$. Therefore, $\tilde{u}\in L^{q}_{\rm loc}(\Om)$, and $\mu[\tilde{u}]\geq \tilde{u}^{q}$ in
$\mathcal{D}'(\Om)$. It is then easy to see that
$$-\Delta_{\mathbb G,\, p}\, \tilde{u}=\tilde{u}^{q} +\mu \quad {\rm in} \quad \mathcal{D}'(\Om)$$
for some nonnegative measure $\mu$ supported on $E$. Moreover, by Theorem \ref{Dmain1} we have
$$\mu(E)\leq C(E) \, {C}_{p,\,\frac{q}{q-p+1}}(E)=0.$$

This gives  $\mu=0$ and thus $\tilde{u}$ solves (\ref{omE}) with $\Om$ in place of $\Om\setminus E$.

Conversely, suppose that every solution to \eqref{omE} can be extended to the whole $\Om$. We will show  ${C}_{p, \, \frac{q}{q-p+1}}(E)=0$
by a contradiction argument.  To this end, suppose that ${C}_{p,\,\frac{q}{q-p+1}}(E)>0$ and consider  the following equation 
\begin{eqnarray}\label{epeq}
\left\{\begin{array}{rcl}
-\Delta_{\mathbb G,\, p}\, u&=& u^q + \epsilon \,
\mu^E \quad {\rm in}~ \Om,\\
u&=&0\quad {\rm on}~ \partial\Om,
\end{array}
\right.
\end{eqnarray}
where $\mu^E$ is the capacitary measure of $E$ with respect to the capacity ${C}_{p,\, \frac{q}{q-p+1}}(\cdot)$
(see \cite{Lu}, \cite{AH}). Note that $\mu^E$ has the following important property (see,  \cite{Lu}, \cite{AH}):
\begin{equation}\label{nonV}
{\mathbf G}_{p}*[{\mathbf G}_{p}*\mu^E]^{\frac{q-p+1}{p-1}} \leq 1\quad {\rm everywhere~on~} {\rm supp}(E).
\end{equation}
Let $K$ be an arbitrary compact subset of $E$ and denote by $\mu_{K}$ the restriction of $\mu^E$ to $K$. We have  
\begin{eqnarray}\label{low}
\int_{\mathbb G}[{\mathbf G}_{p}*\mu_K]^{\frac{q}{p-1}} dx&=&\int_{K}{\mathbf G}_{p}*[{\mathbf G}_{p}*\mu_K]^{\frac{q-p+1}{p-1}}d\mu_K\\
&\leq&\int_{K}{\mathbf G}_{p}*[{\mathbf G}_{p}*\mu^E]^{\frac{q-p+1}{p-1}}  d \mu_{K}\nonumber\\
&\leq &\mu_{K}(K),\nonumber
\end{eqnarray}
where we used \eqref{nonV} in the last inequality.
On the other hand, it follows from the dual definition of capacity, see \eqref{dualB},  that 
\begin{equation}\label{up}
\mu_{K}(K)\leq {C}_{p, \, \frac{q}{q-p+1}}(K)^{\frac{q-p+1}{q}}\norm{{\mathbf G}_{p}(\mu_K)}_{\frac{q}{p-1}(\mathbb G)}.
\end{equation}

 Thus we obtain from \eqref{low} and \eqref{up} that
$$\mu^{E}(K)=\mu_{K}(K)\leq {C}_{p, \, \frac{q}{q-p+1}}(K).$$

Since this holds for all compact sets $K$ by Theorem \ref{Dmain1} we see that the equation \eqref{epeq} is solvable as long as $\epsilon>0$ is chosen small enough.   
But this would give us   a contradiction and hence the proof is complete.
\end{proof}

\section{ Global solutions and Liouville type theorems}\label{secv}
In this section we sketch the proof of Theorem \ref{Gmain1} and Corollary \ref{LiouvilleT}. 
To prove Theorem \ref{Gmain1} one can proceed as in the proof of Theorem \ref{Dmain1} but using Theorem \ref{W-G-R} instead of Theorem \ref{W-G},
and the following global version of Theorem \ref{SU2}. The latest  in turn can be proved as in \cite[Theorem 5.3]{PV1} by approximations and using pointwise estimates for potential theoretic solutions over arbitrarily large balls.

\begin{theorem}%\label{sufficiency}
Let $\om \in \mathcal{M}^{+}(\mathbb G)$, $1<p<M$, and $q>p-1$. Suppose that
\begin{equation*}
{\rm\bf W}^{\infty}_{1,\,p}[({\rm\bf W}^{\infty}_{1,\,p}\om)^{q}]\leq C\,{\rm\bf W}^{\infty}_{1,\,p}\om<+\infty \quad {\rm a.e.,}
\end{equation*}
 where
\begin{equation*}
C \leq \Big(\frac{q-p+1}{q A\max\{1,2^{p'-2}\}}\Big)^{q(p'-1)}\Big(\frac{p-1}{q-p+1}\Big),
\end{equation*}
and $A$ is the constant used in Definition  \ref{DEFEQ}.
 Then there exists a $p$-superhar-monic function $u\in L^{q}_{\rm loc}(\RR^{n})$ such that
\begin{eqnarray*}
\left\{\begin{array}{rcl}
-\Delta_{\mathbb G,\, p}\, u&=&u^{q}+\om,\\
\inf_{\mathbb G}u&=&0,
\end{array}
\right.
\end{eqnarray*}
and for every $x\in \mathbb G$,
$$c_{1}{\rm\bf W}^{\infty}_{1,\,p}\om(x)\leq u(x)\leq c_{2} \,{\rm\bf W}^{\infty}_{1,\,p}\om(x),$$
where the constants $c_{1}, c_{2}$ depend only $p, q$, and $M$.
\end{theorem}

We remark that  in order to show that the existence of a solution $u$ to 
\eqref{global} implies (i) and (ii)  in  Theorem \ref{Gmain1} we need the following analogue of \eqref{mu-mu}: 
\begin{equation*}
\int_{\mathbb{G}}\left[\mathbf{I}_{p}(gd\mu)\right]^{\frac{q}{p-1}}dx \leq C \int_{\mathbb{G}}g^{\frac{q}{p-1}}d\mu,
\end{equation*}
where  $\mu=u^q + \om$. This can be shown to hold for all $g\in L^\frac{q}{p-1}(d\mu)$, $g\geq 0$, with no restriction on the support of $g$ 
by using the lower bound in \eqref{IKM-G}.

Finally, we give a proof of Corollary \ref{LiouvilleT}. 

\begin{proof}[{\bf Proof of Corollary \ref{LiouvilleT}}]
Corollary \ref{LiouvilleT} follows from Theorem \ref{Gmain1} and the fact that for  $\alpha s\geq M$
the Riesz capacity $\dot{{C}}_{\alpha, \, s}(E)=0$ for every compact set $E\subset\mathbb G$. To see the later note that
for any nonnegative measure $\mu$ supported in a ball $B_R(e)$, $R>0$, we have 
$${\rm \bf I}_{\alpha}*\mu(x)=\int_{B_{R}(e)}\frac{1}{d_{cc}(x, y)^{M-\alpha}}d\mu(y)\geq \frac{c\, \mu(B_{R}(e))}{(|x| +R)^{M-\alpha}},$$
where $|x|$ is the homogeneous norm of $x$ (see Sect. \ref{Cgroup}). Thus using the condition $\alpha s\geq M$ and \cite[Corollary 1.6]{Fol} we get 
$\norm{{\rm \bf I}_{\alpha}*\mu}_{L^{\frac{s}{s-1}}(\mathbb G)}=\infty$
unless $\mu$ is identically zero. Therefore, in view of the dual definition of capacity, see \eqref{dualR}, we obtain 
$$\dot{{C}}_{\alpha, \, s}(E)=0$$ 
for every compact set $E\subset\mathbb G$.
\end{proof}


\begin{thebibliography}{xxxxxx}
\bibitem[AH]{AH} D. R. Adams  and   L. I. Hedberg, {\it Function Spaces and Potential
Theory}, Springer-Verlag, Berlin, 1996.
\bibitem[AP]{AP} D. R. Adams  and   M. Pierre, {\it Capacitary strong type estimates in
semilinear problems}, Ann. Inst. Fourier (Grenoble), {\bf 41} (1991), 117--135.
\bibitem[BP]{BP} P. Baras  and   M. Pierre, {\it Crit\`ere d'existence des solutions positives
pour des \'equations semi-lin\'eaires non monotones}, Ann. Inst.  H. Poincar\'e, {\bf 3} (1985),
185--212.
\bibitem[BV1]{BV1}  M. F. Bidaut-V\'eron, {\it Necessary conditions of existence for an elliptic equation with
source term and measure data involving $p$-Laplacian},   Proc.
 2001 Luminy Conf. on Quasilinear Elliptic and Parabolic Equations and Systems.
 Electron. J. Differ. Equ. Conf. {\bf 8}  (2002), 23--34.
\bibitem[BV2]{BV2}  M. F. Bidaut-V\'eron, {\it Removable singularities and existence for a quasilinear equation with 
absorption or source term and measure data}, Adv. Nonlinear Stud.  {\bf 3}  (2003),  25--63.
\bibitem[BCC]{BCC} I. Birindelli, I. Capuzzo Dolceltta and  A. Cutri, {\it Liouville theorems
for semilinear equations on the Heisenberg group}, Ann. Inst.  H. Poincar\'e {\bf 14} (1997),
295--308.
\bibitem[COV]{COV} C. Cascante, J. M. Ortega,  and   I. E. Verbitsky, {\it Nonlinear
potentials and two weight trace inequalities for general dyadic and radial kernels},
Indiana Univ. Math. J., {\bf 53} (2004), 845--882.
\bibitem[Cho]{Cho} W. L. Chow, {\it  $\ddot{U}ber$ systeme von linearen partiellen Differentialgleichungen
erster Ordnung}, Math. Annalen, {\bf 117} (1939), 98--105.
\bibitem[Chr]{Chr} M. A. Christ, {\it A $T(b)$ theorem with remarks on analytic capacity and the Cauchy integral}, Colloq. Math. 
{\bf 60/61} (1990), 601--628.
\bibitem[CLL]{CLL} W. Cohn, G. Lu, and S. Lu, {\it Higher order Poincar\'e inequalities associated with linear operators on stratified
groups and applications} Math. Z. {\bf 244} (2003), 309--335.
%\bibitem[CW]{CW} R. Coifman and G. Weiss, {\it Extensions of Hardy spaces and their use in analysis},
%Bull. Amer. Math. Soc. {\bf 83} (1977), 569--645.
\bibitem[DMOP]{DMOP} G. Dal Maso, F. Murat, A. Orsina, and A. Prignet,
{\it Renormalized solutions
of elliptic equations with general measure data}, Ann. Scuol. Norm. Super. Pisa (4) {\bf 28}
(1999), 741--808.
\bibitem[Fol]{Fol} G. B. Folland, {\it Subelliptic estimates and function spaces on
nilpotent Lie groups}, Arkiv f\"or Mat. {\bf 13} (1975), 161--207.
\bibitem[GL]{GL} N. Garofalo and E. Lanconelli, {\it Existence and nonexistence
results for semilinear equations on the Heisenberg group}, Indiana Univ. Math. J.
{\bf 41} (1992), 71--97.
\bibitem[HW]{HW} L. I. Hedberg and Th. H. Wolff, {\it Thin sets in nonlinear potential theory}, Ann. Inst. 
Fourier (Grenoble) {\bf 33} (1983), 161--187.
\bibitem[HKM]{HKM} J. Heinonen, T. Kilpel\"ainen,  and   O. Martio,  {\it Nonlinear Potential
Theory of Degenerate Elliptic Equations}, Oxford Univ. Press, Oxford, 1993.
%\bibitem[Hor]{Hor} L. H\"ormander, {\it Hypoelliptic second order differential
%equations}, Acta Math. {\bf 119} (1967), 147--171.
\bibitem[Kil]{Kil} T. Kilpel\"ainen, {\it $p$-Laplacian type equations involving measures}, Proceedings of the 
International Congress of Mathematicians, Vol. III (Beijing, 2002), 167--176. 
\bibitem[KM1]{KM1} T. Kilpel\"ainen  and   J. Mal\'y,  {\it Degenerate elliptic equations
with measure data and nonlinear potentials},  Ann. Scuola Norm. Super. Pisa, Cl.
 Sci. {\bf 19} (1992), 591--613.
\bibitem[KM2]{KM2} T. Kilpel\"ainen  and   J. Mal\'y, {\it The Wiener test and potential estimates
for quasilinear elliptic equations}, Acta Math. {\bf 172} (1994), 137--161.
\bibitem[KR]{KR} A. Kor\'anyi and H. M. Reimann, {\it Foundations for the theory of quasiconformal mappings on the Heisenberg group},
Adv. Math. {\bf 111} (1995),  1--87.
%\bibitem[MP]{MP} E. Mitidieri and S. I. Pohozaev, Nonexistence of positive solutions for quasilinear elliptic
%problems on $\RR^n$, Pro. Steklov Inst. Math. {\bf 227} (1999), 1--32.
\bibitem[Lu]{Lu} G. Lu, {\it Potential analysis on Carnot groups, Part II: Relationship between Hausdorff
measures and capacities}, Acta Math. Sinica, English Series, {\bf 20} (2004), 25--46.
%\bibitem[Mik]{Mik} P. Mikkonen, {\it On the Wolff potential and quasilinear elliptic equations
%involving measures}, Ann. Acad. Sci. Fenn., Ser AI, Math. Dissert. {\bf 104} 1996, 1--71.
\bibitem[NSW]{NSW} A. Nagel, E. M. Stein, and S. Wainger, {\it Balls and metrics defined
by vector fields I: Basic properties}, Acta Math. {\bf 155} (1985), 103--147.
\bibitem[PV1]{PV1} N. C. Phuc and I. E. Verbitsky,
{\it Quasilinear and Hessian equations of Lane--Emden type}, Ann.  Math. {\bf 168} (2008), 859--914.
\bibitem[PV2]{PV2} N. C. Phuc and I. E. Verbitsky,
{\it Singular quasilinear and Hessian equations and inequalities}, J. Funct. Analysis {\bf 256} (2009), 1875--1906.
\bibitem[PVe]{PVe} S. Pohozaev and L. V\'eron, {\it Nonexistence results of
solutions of semilinear differential inequalities on the the Heisenberg group},
Manuscripta Math. {\bf 102} (2000), 85--99.
\bibitem[Ra]{Ra} P. K. Rashevsky, {\it Any two points of a total nonholonomic space may be connected by an
admissible line}, Uch. Zap. Ped. Inst. Liebknecht, Ser. Phys. Math., (Russian) {\bf 2} (1938), 83--94.
%\bibitem[RS]{RS} L. P. Rothschild and E. M. Stein, {\it Hypoelliptic differential
%operators and nilpotent groups}, Acta. Math. {\bf 137} (1976), 247--320.
\bibitem[SW]{SW} E. T. Sawyer and R. L. Wheeden, {\it Weighted inequalities
for fractional integrals on Euclidean and homogeneous spaces}, Amer. J. Math.
{\bf 114} (1992), 813--874.
\bibitem[Tu]{Tu} B. O. Turesson, {\it Nonlinear potential theory and weighted Sobolev spaces}, Lecture Notes Math.,
{\bf 1736} (2000), 1--173.
\bibitem[TW]{TW} N. S. Trudinger  and   X. J. Wang, {\it On the weak continuity
of elliptic operators and applications to potential theory}, Amer. J. Math.
{\bf 124} (2002), 369--410.
\bibitem[Va]{Va} V. S. Varadarajan, {\it Lie Groups, Lie Algebras, and Their Representations},
Springer-Verlag, New York-Berlin-Heidelberg-Tokyo, 1974.
\bibitem[VSC]{VSC} N. Th. Varopoulos, L. Saloff-Coste, and T. Coulhon, {\it Analysis and Geometry on Groups},
Cambridge Univ. Press, 1992.
\end{thebibliography}
\end{document}